\documentclass[journal,10pt,onecolumn,draftclsnofoot]{IEEEtran}

\usepackage{etoolbox}
\newtoggle{arxiv}

\toggletrue{arxiv}
\iftoggle{arxiv}{%
	\newcommand{\resub}[1]{{#1}}
	\newcommand{\ressub}[1]{{#1}}
}{%
	\newcommand{\resub}[1]{{#1}}
	\newcommand{\ressub}[1]{{\color{blue}#1}}
}

\usepackage{lineno}
\usepackage{graphics} % for pdf, bitmapped graphics files
\usepackage{epsfig} % for postscript graphics files
\usepackage{amsmath} % assumes amsmath package installed
\usepackage{amssymb}  % assumes amsmath package installed
\usepackage{color} 
\usepackage{cite}
\usepackage{algorithm,algpseudocode}
\usepackage{amsmath}
\usepackage{mathtools}
\usepackage{accents}
\usepackage{parskip}
\usepackage{datatool}

\usepackage{amsthm}
\usepackage{cases}

\usepackage[normalem]{ulem}

\let\OLDthebibliography\thebibliography
\renewcommand\thebibliography[1]{
	\OLDthebibliography{#1}
	\setlength{\parskip}{0pt}
	\setlength{\itemsep}{0pt plus 0.3ex}
}

\DTLsetseparator{=}
\newcommand{\sval}[2]{
	\DTLifdbexists{#1}{}{\DTLloaddb[noheader, keys={thekey,thevalue}]{#1}{#1.txt} }
	\DTLfetch{#1}{thekey}{#2}{thevalue}
}

\graphicspath{{./svg}{./figs}}

\newtheoremstyle{pplain}% name of the style to be used
{3pt}% measure of space to leave above the theorem. E.g.: 3pt
{3pt}% measure of space to leave below the theorem. E.g.: 3pt
{\normalfont}% name of font to use in the body of the theorem
{0pt}% measure of space to indent
{\itshape}% name of head font
{:}% punctuation between head and body
{2pt}% space after theorem head; " " = normal interword space
{}% 

\theoremstyle{pplain}
\newtheorem{theorem}{Theorem}

\theoremstyle{pplain}
\newtheorem{remark}{Remark}

\theoremstyle{pplain}
\newtheorem{lemma}{Lemma}

\theoremstyle{pplain}
\newtheorem{lemmaa}{Lemma}[section]

\theoremstyle{pplain}
\newtheorem{definition}{Definition}

\theoremstyle{pplain}

\theoremstyle{pplain}

\theoremstyle{pplain}
\newtheorem{corollary}{Corollary}

\theoremstyle{pplain}

\theoremstyle{pplain}
\newtheorem{innercustomthm}{}
\newenvironment{pbm}[1]
{\renewcommand\theinnercustomthm{#1}\innercustomthm}
{\endinnercustomthm}

\newcommand{\R}{\mathbb{R}}
\newcommand{\Ss}{\mathbb{S}}
\newcommand{\Sp}{\mathbb{S}_+}
\newcommand{\Spp}{\mathbb{S}_{++}}
\newcommand{\G}{\mathcal{G}}

\newcommand{\E}{\mathcal{E}}
\newcommand{\V}{\mathcal{V}}

\newcommand{\tp}{{^\intercal}}
\newcommand{\Z}{\mathcal Z}

\DeclareMathOperator{\tr}{tr}
\DeclareMathOperator{\rank}{rank}
\DeclareMathOperator{\svd}{svd}
\DeclareMathOperator{\diag}{diag}
\DeclareMathOperator{\vecop}{vec}

\DeclareMathOperator{\find}{find}

\newcommand{\tnn}[2]{\eta_{#2}(#1)}
\newcommand{\kfn}[2]{\left\Vert #1 \right\Vert_{\left\lceil #2 \right\rceil}}
\DeclareMathOperator{\minimize}{minimize}
\DeclareMathOperator{\argmin}{argmin}
\DeclareMathOperator{\subjecto}{subject\,to}

\DeclareMathOperator{\sgn}{sgn}

\newcommand{\A}{\mathcal{A}}
\newcommand{\B}{\mathcal{B}}

\newcommand{\inner}[2]{\left \langle #1, #2 \right \rangle}

\newcommand{\fix}{\check}
\newcommand{\sta}{\bar}
\newcommand{\tar}{\tilde}

\newcommand{\perturbation}{intervention}

\newcommand{\erdos}{Erd\H{o}s-R\'enyi}
\newcommand{\ieeebus}{IEEE \mbox{14-bus} system}

\DeclareMathOperator{\Pone}{\mathcal{P}_{1}}
\DeclareMathOperator{\Ponedc}{\mathcal{P}_{1-\mathrm{DN}}}
\DeclareMathOperator{\Ponesub}{\mathcal{P}_{1-\mathrm{SUB}}}
\DeclareMathOperator{\Ponestd}{\mathcal{P}_{1-\mathrm{STD}}}

\DeclareMathOperator{\Ptwo}{\mathcal{P}_{2}}
\DeclareMathOperator{\Ptwodc}{\mathcal{P}_{2-\mathrm{DN}}}
\DeclareMathOperator{\Ptwosub}{\mathcal{P}_{2-\mathrm{SUB}}}

\newcommand{\ommit}[1]{{}}

\newlength{\bibitemsep}
\newlength{\bibparskip}
\let\oldthebibliography\thebibliography
\renewcommand\thebibliography[1]{%
	\oldthebibliography{#1}%
	\setlength{\parskip}{\bibitemsep}%
	\setlength{\itemsep}{\bibparskip}%
}

\begin{document}

\title{
	Network Design for Controllability Metrics
}

\author{
	Cassiano~O.~Becker,~S\'{e}rgio~Pequito,~George~J.~Pappas~and~Victor~M.~Preciado
\thanks{
	 This work was supported in part by the National Science Foundation, grant CAREER-ECCS-1651433 and in part by CAPES, Coordena\c{c}\~{a}o de Aperfei\c{c}oamento de Pessoal de N\'{i}vel Superior - Brasil.
}
\thanks{
	Cassiano O. Becker (cassiano@seas.upenn.edu), George J. Pappas (pappasg@seas.upenn.edu) and Victor M. Preciado (preciado@seas.upenn.edu) are with the Department of Electrical and Systems Engineering, University of Pennsylvania - 200 South 33rd Street, Philadelphia, PA 19104.
	S\'{e}rgio Pequito (goncas@rpi.edu) is with the Department of Industrial and Systems Engineering, Rensselaer Polytechnic Institute - CII 5007, 110 8th Street,  Troy, NY 12180-3590.
}% <-this % stops a space
%\thanks{Manuscript received Month Day, Year; revised Month Day, year.}
}

%\markboth{IEEE Transactions on Control of Network Systems,~Vol.~XX, No.~YY, Month~Year}%

\markboth{}%
{
	%submitted to IEEE Transactions on Control of Network Systems
}

\maketitle

\newcommand{\cout}[1]{\textcolor{red}{\sout{#1}}}
\newcommand{\cin}[1]{\textcolor{magenta}{#1}}

\begin{abstract}
In this paper, we consider the problem of tuning the edge weights of a networked system described by linear time-invariant dynamics. We assume that the topology of the underlying network is fixed and that the set of feasible edge weights is a given polytope. In this setting, we first consider a feasibility problem consisting of tuning the edge weights such that certain controllability properties are satisfied. The particular controllability properties under consideration are (\emph{i}) a lower bound on the smallest eigenvalue of the controllability Gramian, and (\emph{ii}) an upper bound on the trace of the Gramian inverse. In both cases, the edge-tuning problem can be stated as a feasibility problem involving bilinear matrix equalities, which we approach using a sequence of convex relaxations. Furthermore, we also address a design problem consisting of finding edge weights able to satisfy the aforementioned controllability constraints while seeking to minimize a cost function of the edge weights, which we assume to be convex. Finally, we verify our results with numerical simulations over many random network realizations, as well as with an IEEE 14-bus power system topology.
\end{abstract}


\IEEEpeerreviewmaketitle

%==========================================
\section{Introduction}

\IEEEPARstart{M}{any} technological, biological, chemical, and social systems can be modeled as large ensembles of dynamical units connected via an intricate pattern of interactions~\cite{FB-LNS}.
From an engineering perspective, we are interested in efficiently steering the dynamics of these complex systems via external actuation.
In this direction, control theory provides us with the notion of \emph{controllability} to decide whether a given system can be steered towards an arbitrary state~\cite{hespanha2009linear}. Furthermore, the so-called \emph{controllability Gramian} of a system, which implicitly depends on the system's dynamics and the configuration of its actuators, can be used to quantify the energy required to steer the system, assuming the system is controllable~\cite{hespanha2009linear}.
Leveraging these notions, several papers have recently focused on the problem of optimally allocating actuators throughout the network under several performance metrics~\cite{
	clark2012,
	6580798,
	summers2016actuator,
	summers2014submodularity,
	pequito2017robust,
	olshevsky2014minimal,
	PequitoJournal,
	tzoumas2016minimal,
	pequito2016minimum,
	enyioha2014controllability
}.

In some scenarios, instead of designing the location of external actuators, one may consider the alternative problem of modifying the network's dynamics given a fixed configuration of actuators. 
For example, in power systems, one can tune the electrical parameters of the transmission lines using, for example,
flexible AC transmission system (FACTS) devices~\cite{ilic2000dynamics,zhang2012flexible}.
Also, in multi-agents networks, the interactions between agents can usually be modified to achieve a particular objective~\cite{xiao2004fast}.
For instance, in leader-follower multi-agent networks, one may consider the scenario where both the communication topology and the location of the external actuators are fixed. Then, one can seek a set of edge weights (e.g., the agents' update rules) such that the average and/or worst-case energy required to drive the state of the network satisfies certain bounds. In this regard, the present work first considers the feasibility problem of finding the edge weights of a linear networked system such that certain bounds on controllability metrics are satisfied. Secondly, we address the design problem of finding edge weights able to satisfy the aforementioned bounds while seeking to minimize a cost function of the edge weights, which we assume to be convex. In particular, we consider a $1$-norm sparsity-promoting cost function aiming to penalize the number of edges whose weights are modified in the resulting design.

%---------------------------------------------------------
\subsubsection{Related Work} 

In recent years, the problem of designing systems to satisfy certain controllability metrics has mostly focused on finding optimal actuator configurations, i.e., the location of those nodes to be externally actuated by control inputs~\cite
{
	clark2012,6580798,
	summers2016actuator,
	summers2014submodularity,
	pequito2017robust,
	olshevsky2014minimal,
	PequitoJournal,
	tzoumas2016minimal,
	pequito2016minimum,
	enyioha2014controllability
}.
In addition, a considerable amount of research has been dedicated to understanding how the network topology impacts control performance~\cite{
	pasqualetti2014controllability,
	bianchin2015role,	
	aguilar2016almost,
	aguilar2015graph,
	parlangeli2012reachability,
	notarstefano2013controllability,
	chapman2014controllability,
	tanner2004controllability,
	pequito2017robust,
	roy2019controllability,
	zhao2017discrete
}.
In particular, \cite{zhao2017discrete} establishes necessary and sufficient graph-theoretical conditions for a discrete-time networked system to exhibit a diagonal controllability Gramian. 
In~\cite{
	zhao2017gramian
}, the authors characterize the minimum input energy required to transfer a discrete-time dynamical system with bilinear dynamics from the origin to a desired state.
The work in~\cite{
	bianchin2016observability
} proposes the notion of observability radius, which measures how much the parameters of a dynamical system can be perturbed before the system becomes unobservable.
In a similar direction, the work in~\cite{siami2018growing} investigates the effect of adding network edges to improve spectral performance metrics for the case of consensus dynamics over networks.
\resub{
	% R2-4 [literature review]
	More generally, the works in~\cite{shafi2011graph,torres2018dominant,preciado2016distributed,sun2018weighted,hassan2017topology} investigate design problems that seek to optimize network dynamical properties such as the dominant eigenvalue of the system matrix, with applications to virus spread and wireless control networks appearing in~\cite{preciado2013optimal,preciado2014optimal,pajic2011wireless,wan2008designing}.%
}

\resub{	
	% R1-0 [contribution]
	%% RESUB
	The present paper extends previous work by the authors in~\cite{becker2017network} through several contributions. Specifically, in this paper we: 
	\emph{(i)}~\ressub{address the discrete-time case, in which} the discrete Lyapunov equation introduces higher-degree products in its decision variables and requires new transformation steps for its treatment;
	\emph{(ii)}~provide an analysis of the conditions under which stability of the designed system is assured;
	\emph{(iii)}~consider cost functions over edge weights, which can be used to promote solutions with higher sparsity in edge modifications;
	\emph{(iv)}~propose a convex relaxation approach, which enables a more detailed analysis of convergence;
	\emph{(v)}~consider average controllability as an additional controllability metric; and 
	\emph{(vi)}~present comprehensive computational experiments to illustrate the above aspects.
}

%---------------------------------------------------------
\subsubsection{Structure and contributions of the paper}

The rest of the paper is organized as follows. In Section~II, we formalize both the network feasibility and the network design problems, in which we are tasked with tuning the weights of the edges in a given network in order to satisfy certain controllability metrics. Specifically, we consider two metrics: (\emph{i}) the worst-case control energy, which is related to the smallest eigenvalue of the Gramian, and (\emph{ii}) the average energy required to drive the system, which is related to the trace of the Gramian inverse. In Section~III, we provide a detailed description of the strategy followed to solve both problems. In particular, we cast both the feasibility and the design problems into nonlinear optimization programs with quadratic bilinear terms, which are, in general, computationally hard to solve.
We approach these optimization problems by lifting the space of variables and adding a rank constraint on a matrix whose entries depend affinely on the decision variables. We then propose a sequence of convex problems to relax this rank constraint using a truncated nuclear norm.
In Section~IV, we illustrate the validity of our results via computational experiments on random graphs, as well as 
%an  $\ell_1$
\resub{a \mbox{$1$-norm}}
 sparsity-promoting design problem considering the \ieeebus.
We conclude \resub{and} enumerate some possibilities for future work in Section~V.

%---------------------------------------------------------
\subsubsection*{Notation}

We denote by $[X]_{i,j}$ the entry at the $i$-th row and $j$-th column of the matrix $X \in \R^{m \times n}$. The transpose of X is written as $X^\tp$. 
The $n \times n$ identity matrix is denoted by $I_n$. 
%The $n$-dimensional vector of zeros is denoted by $0_n$.
The operator $\diag(a_1,\ldots, a_n)$ returns a diagonal matrix having $a_1,\ldots,a_n$ as entries in its diagonal.
The inner product between two matrices $X,Y \in \R ^{m \times n}$ is given by 
$\left < X, Y \right > = \tr\{ X^\tp Y\}$, 
\resub{
	% R1-2 [notation]
	where $\tr\{X^\tp Y\}=\sum_{i=1}^n [X^\tp Y]_{i,i}$ denotes the trace operator.
}
The $1$-norm of a matrix $X \in \R^{m \times n}$ is defined as the $\ell_1$-norm of its vectorization, i.e., $\|X\|_1 = \| \!\vecop(X)\|_1$.
\resub{
	% R1-2 [notation]
	Likewise, the $0$-norm of a matrix is defined as the $\ell_0$-quasi-norm of its vectorization, i.e., the number of nonzero entries. The infinity norm of~$X$ is defined as $\| X\|_\infty= \max_{i,j} [X]_{i,j}$.
}
\resub{
	% R1-2 [notation]
	The nuclear norm of $X$ is defined in terms of its singular values~$\sigma_i(X)$, $i=1,\ldots, \min\{m,n\}$, as 
	$\|X\|_\ast = \sum_{i=1}^{\min\{m,n\}} \sigma_i(X)$.
}
\resub{
	% R1-2 [notation]
	The operator norm of $X$ is denoted by $\|X\|$ and computed as $\|X\|=\sigma_1(X)$, the largest singular value of $X$.
}
%We denote by $\Sp^n$ (resp., $\Spp^n$) the set of symmetric positive semi-definite (resp.,  definite) matrices. Correspondingly, the semidefinite partial ordering is denoted $X \succeq Y$ (resp., $X \succ Y$) when $X-Y \succeq 0$ (resp., $X-Y \succ 0$).
\resub{We denote by $\Ss^n$ the set of symmetric matrices of dimension $n$}.
Likewise, $\Sp^n$ (resp., $\Spp^n$) is the set of symmetric positive semidefinite (resp.,  definite) matrices. Correspondingly, the semidefinite partial ordering is denoted $X \succeq Y$ (resp., $X \succ Y$) when $X-Y \succeq 0$ (resp., $X-Y \succ 0$).
\resub{
	% R1-2 [notation]
	A set~$\mathcal S \subset \R^m$ is a spectrahedron \cite[Def. 2.6]{blekherman2012semidefinite} if it can be represented in the form $\mathcal S =\{ (x_1, \ldots, x_m) \in \R^m : Q_0 + \sum _{i=1}^m Q_i x_i \succeq 0\}$, for $Q_0, \ldots, Q_m \in \Ss^n$.
}
\resub{A proper algebraic variety~$\mathcal V \subset \R^n$ is the set of common zeros of a finite number of nonzero polynomials in $n$ variables.}
\resub{
	% R1-2[notation]
	%The signum function, denoted $\sgn(x)$, evaluates to $1$ if $x>0$, to $-1$ if $x<0$, and to $0$ if $x=0$.
}

%==============================================	
\section{Problem Formulation}
\label{sec:pr_state}

\newcommand{\WrT}{{W_{r,T}}}
\newcommand{\Wr}{W_{r}^{\infty}}
\newcommand{\WrTinv}{(\WrT)^{-1}}
\newcommand{\Wrinv}{(\Wr)^{-1}}

Consider a networked system following a discrete-time linear time-invariant dynamics, described by
\begin{align}
x(k+1) = A(\G)x(k) + Bu(k),
\label{eq:disc_dyn}
\end{align}
where $x(k)\in \R^n$ denotes the vector of states and $u(k)\in \R^m$ is the vector of inputs at instant $k$. The sparsity pattern of the state matrix~$A(\G)\in \R^{n \times n}$ is constrained by a directed \emph{interdependency graph}~$\G = \left (\V, \E  \right )$ defined by a set of nodes~$\V=\{1,\ldots,n\}$ and a set of edges~$\E \subseteq \V \times \V$, such that $[A(\G)]_{i,j} \in \R$ if the edge~$(j,i)\in \E$, and $[A(\G)]_{i,j} = 0$ if $(j,i)\notin \E$. Also, the input matrix~$B \in \R^{n \times m}$ is such that $[B]_{i,l}\neq 0$ if the external input signal $[u(k)]_l$ directly influences~$[x(k+1)]_i$, and $[B]_{i,l}= 0$ otherwise.

Next, consider the problem of driving the state of the network from a given initial state $x_0 \equiv x(0)$ to a desired target state $x_T \equiv x(T)$ within a time horizon $T>0$, by designing a sequence of inputs~$u(k)$ for $k \in \{0,1,\ldots, T-1\}$. If any $x_T\in\R^n$ is attainable from $x_0=0_n$ within a time horizon~
%$T=n$,
\resub{$T$},
then the system~\eqref{eq:disc_dyn} is said to be \emph{reachable}, which we refer to $(A(\G),B)$ being reachable. Furthermore, it is known that the minimum input control energy to steer the system to a desired final state $x_T$ from $x_0=0_n$ is given by~\cite{hespanha2009linear}
\begin{align}
\label{eq:min_energy}
J(T,x_T) \coloneqq x_T^\tp \WrTinv x_T,
\end{align}
where $\WrT$ is called the \emph{finite-horizon reachability Gramian}, defined as
%\begin{align*}
$
\WrT \coloneqq \sum_{k=0}^{T-1} A(\G)^kBB^\tp (A(\G)^\tp)^k .
$
%\end{align*}
The \emph{infinite-horizon reachability Gramian} is then obtained as the limit $\Wr \coloneqq \lim_{T \rightarrow \infty} \WrT$. This Gramian is positive definite, and can be computed as the (unique) solution to the \mbox{discrete-time} Lyapunov equation
\begin{align}
\label{eq:disc_lyap}
A(\G)\Wr A(\G)^\tp - \Wr + BB^\tp=0.
\end{align}
when the system is reachable 
\resub{%
	% R1-3, R2-0 [stability]
	and $A(\G)$ is stable}~\cite{hespanha2009linear}. 
%-------------------------------------------------
\subsection{Reachability Metrics}

\newcommand{\spec}{\mathcal W_\theta}

We focus on two metrics related to the reachability Gramian to quantify the minimum input energy to drive the system~\cite{muller1972analysis,pasqualetti2014controllability,zhao2017gramian}.

%.....................................................
\paragraph{Worst-case minimum input energy}

Because $W_r^\infty$ is (symmetric) positive definite when the system is reachable, its eigenvalues $\lambda_1\le \ldots \le \lambda_n$ are positive real numbers, with corresponding eigenvectors $v_i$ for $i=1,\ldots,n$. It turns out that the final state $x_T$ satisfying $\|x_T\|_2=1$ requiring the largest minimum input energy to be reached from $x_0=0_n$ is given by the (normalized) eigenvector $v_1$. The energy required to drive the state from the origin towards $v_1$ within an infinite horizon is equal to $\lambda_1^{-1}$, which we call the \emph{worst-case minimum input energy}.
Therefore, if we require the worst-case minimum input energy to be less than or equal to a desired value~$\tar \lambda^{-1} \resub{\,>0}$, then the reachability Gramian must satisfy the following semidefinite constraint:
\begin{align}
\label{eq:req_min_lam}
\Wr - \tar \lambda I_n \succeq 0.
\end{align}

%.....................................................
\paragraph{Average minimum input energy}

The expected energy required to steer the system from the origin towards a random final state uniformly distributed over the unit sphere is equal to
$\frac{1}{n} \tr \{\Wrinv\}$~\cite{muller1972analysis},
which we call the \emph{average minimum input energy}.
% Similarly
\resub{In a manner similar}
 to the worst-case minimum input energy metric, we can constrain the average minimum input energy to be upper-bounded by a target value~$\tar \tau \resub{\;< \infty}$ via the condition
\begin{align}
\label{eq:req_trace_inv}
n\tar \tau - \tr \{\Wrinv\} \geq 0,
\end{align}
which is also representable by a semidefinite constraint over $\Wr$ (see Lemma~\ref{lem:tr_inv} in the Appendix).

In what follows, we will refer to the aforementioned reachability constraints on $\Wr$ by the set membership condition
\begin{align}
\label{eq:set_spec}
\Wr \in \spec,
\end{align}
where $\spec$ is a convex set (more precisely, a spectrahedron) defined by  constraints~\eqref{eq:req_min_lam} and/or \eqref{eq:req_trace_inv}, and indexed by the parameters in~$\theta = (\tar \lambda, \tar \tau)$.

%-----------------------------------------------------------------------
\subsection{Network Design for Reachability}

\newcommand{\poly}{\mathcal D}

As previously mentioned, we consider the problem of tuning the edge weights of a given network in order to satisfy certain minimum control energy requirements (either in worst-case or in average).
In particular, we assume that we are able to add a matrix $\Delta(\G) \in \R^{n \times n}$ to the state matrix $A(\G)$, such that $\Delta(\G)$ presents the same sparsity pattern as the interdependency graph, i.e., $ [\Delta(\G)]_{i,j}=0$ for $(j,i)\notin \mathcal E$. After this addition, the dynamics of the network becomes
\begin{align}
\label{eq:disc_dyn_pert}
x(k+1)=[A(\G)+\Delta(\G)]x(t)+Bu(t).
\end{align}
Furthermore, we may require that $\Delta(\G)$ be contained in a given polytope~$\poly$ encoding acceptable limits for its entries. For example, we can impose
upper and lower bounds of the form~$[\Delta(\G)]_{i,j}\in [\iota_{i,j},\upsilon_{i,j}]$ for $(j,i)\in \mathcal E$ in the design problem. 
Subsequently, we consider the model described by~\eqref{eq:disc_dyn_pert} and address the following two problems\footnote{For compactness of notation, we will denote $A(\G)$, $\Wr$, and $\Delta(\G)$ simply by $A$, $W$, and $\Delta$, respectively, in the rest of the paper.}.

%..........................................................................................
\subsubsection{Feasible Design for Reachability Metrics}

We seek an addition~$\Delta \in \poly$ such that the resulting reachability Gramian~$W \in \Spp^n$ satisfies~$W \in \spec$. This can be posed as the following feasibility problem:

\newcommand{\schurconst}{| \lambda_i (A+\Delta )| < 1,  \;\,  i=1,\ldots,n}

\begin{pbm}{$\Pone$}[Feasible Design for Reachability Metrics]\label{pbm:pone} 
	Given the interdependency graph~$\G$, with $(A,B)$ reachable, we would like to 
	\begin{align}
	\find \qquad   &\Delta \in \R^{n\times n},\; W \in \Spp^n \notag \\ 
	\subjecto \qquad &W \in \spec, \label{eq:spec}\\
	&\Delta \in \poly, \label{eq:poly}\\
	(A+\Delta)&W (A+\Delta)^\tp
	- W + BB^\tp=0, \label{eq:lyap_delta}\\
	&\resub{\schurconst{}}, \label{eq:schur_const}
	\end{align}
\end{pbm}
\resub{
	% R1-4, R2-0 [stability]
%where the last constraint follows from the discrete-time Lyapunov equation associated with~\eqref{eq:disc_dyn_pert}.
where constraint~\eqref{eq:lyap_delta} arises from the discrete-time Lyapunov equation associated with~\eqref{eq:disc_dyn_pert}, and 
constraint~\eqref{eq:schur_const} enforces the stability of the designed system. }

%\begin{remark}
	%In general, for our results to hold, the matrix $\Delta$ could be imposed any sparsity pattern. However, in this paper, we seek to find a $\Delta$ having the same sparsity pattern as the one implied by the given interdependency graph $\G$.
	% However, in this paper, we seek to find a $\Delta$ having the same sparsity pattern as the one implied by the given interdependency graph $\G$.
%\end{remark}

\begin{remark}
	\resub{
		% R2-2 [partial design]
		Partial design, allowing only a subset of the edge weights to be modified, can be performed by imposing additional constraints $[\Delta]_{ i, j} = 0$ for the edges~$(j,i)$ that cannot be affected by the design procedure.%
	 }
\end{remark}
As we will show in the next section, this feasibility problem can be addressed using a sequence of convex relaxations. This problem also lays the foundation to our second problem, described next.

%.......................................................................................
\subsubsection{Design for Reachability with Structural Penalties}

In this formulation, we introduce an optimization objective that penalizes entries of $\Delta$ with large magnitudes, while meeting the reachability requirements on $W$ and structural constraints on $\Delta$. In particular, aiming at penalizing the number of edges modified, we consider the $1$-norm penalty over the entries of $\Delta$ as our cost function.
\resub{
	% R2-3 [sparsity]
	%% RESUB
	%Using an L1-norm based cost for sparse design is common in the statistical signal processing and perhaps optimization communities, but is less so in the controls literature. The authors may want to discuss the motivation for the metric briefly.	
	%Convex relaxation of the cardinality function
	The $1$-norm behaves as a convex envelope to the $0$-norm (i.e., the number of non-zero entries in the matrix), and has found wide use in the signal processing and optimization literature~\cite{donoho2006compressed, recht2010guaranteed, hastie2015statistical}. In control systems problems, it has been successfully applied to promote sparsity in control architectures, for instance, in~\cite{dorfler2014sparsity,lin2013design}.
}

\begin{pbm}{$\Ptwo$}[Design for Reachability with Structural Penalties]\label{pbm:ptwo}
	 Given an interdependency graph~$\G$ and a reachable system~$(A,B)$, find a structural addition~$\Delta$ seeking to
	%,
	\begin{align*}
	\underset{\substack{\Delta \in \R^{n\times n}\\ W \in \Spp^n}}{\minimize}\qquad   & \| \Delta \|_1\\
	\subjecto \qquad & 
	\eqref{eq:spec},
	\eqref{eq:poly},
	\eqref{eq:lyap_delta}\resub{\text{ and }
	\eqref{eq:schur_const}.} 
	\end{align*}
\end{pbm}

As will be described in Section~\ref{sec:penalty}, this problem can be addressed by a sequence of convex relaxations involving an additive penalty term over the 1-norm of $\Delta$, whose limiting value is obtained by a 
%regularization path~\cite{blomberg2014approximate}. 
procedure called \emph{regularization path}~\cite{blomberg2014approximate}.

\begin{remark}
More generally, in \ref{pbm:ptwo}, we could consider a cost function having individual weights over the entries of $\Delta$. For simplicity, in this paper we consider all entries to have unit weight.
\end{remark}

%%%%%%%%%%%%%%%%%%%%%%%%%%%%%%%%%%%%
%%%%%%%%%%%%% BODY %%%%%%%%%%%%%%%%%%%
%%%%%%%%%%%%%%%%%%%%%%%%%%%%%%%%%%%%

%=================================
\section{Design for Reachability Algorithm}

%% TODO: simulations: WE CHECK FOR ALG. VARIETY and DID NOT OCCUR (both kronecker, uncontrollable), as expected

\newcommand{\varA}{\mathcal V}
\newcommand{\varW}{\mathcal V_0}
\newcommand{\varC}{\mathcal V_1}
\global\long\def\vc{\mathrm{vec}}

\newcommand{\almostsurely}{\emph{almost surely}}

\resub{
	%% R1-0 [roadmap]
	In this section, we propose a computational procedure to address~$\Pone$ and $\Ptwo$. We begin by providing preliminary analyses of the Lyapunov equation~\eqref{eq:lyap_delta} and of the stability constraint~\eqref{eq:schur_const}. We show that the Lyapunov equation constraint can be transformed into a rank constraint, and that its solution will imply the stability of $A + \Delta$ \almostsurely{}.
	Then, we solve $\Pone$ by handling the rank constraint through a sequence of convex problems with guaranteed convergence.  
	Subsequently, we address $\Ptwo$ by computing a regularization path over a weight parameter that controls the sparsity of the generated solutions. 
	%\textcolor{red}{Finally, we provide an analysis of the conditions under which a point generated by the sequence of convex problems attains global optimality.}
}

%----------------------------------------------------
\subsection{Stability from a positive solution to the Lyapunov Equation}

\resub{
	%% R2-0, R1-3 [stability]
	%Even though a stability constraint on $A + \Delta$ is not explicitly enforced for $\Pone$, the designed system can be shown to be Schur stable for all $\Delta \in \poly$ except for a degenerate set with Lebesgue measure zero, as established in the next theorem. 	
	In this section, we show that constraint~\eqref{eq:schur_const} is satisfied \almostsurely{} by all $\Delta \in \poly$ that satisfy the Lyapunov constraint in~\eqref{eq:lyap_delta}. Following methodologies similar to~\cite{davison1973properties,shields1976structural,dion2003generic,menara2018structural}, we formalize this result in the next theorem.

	\begin{theorem}[Stability of the designed system]	
		\label{thm:stab_lyap_generic}
		For a solution~$(W,\Delta)$ to \eqref{eq:lyap_delta} with $W\succ 0$, if the original system~$(A,B)$ is reachable, then the system~$A+\Delta$ will be stable for any $\Delta \in \poly \setminus \varA$, where $\varA$ is a set with Lebesgue measure zero.
	\end{theorem}
	\begin{proof}
		Applying Lemma~\ref{thm:generality_uniqueness} from the Appendix for the matrix~$A+\Delta$, we have that a solution~$W$ to \eqref{eq:lyap_delta} exists and is unique for all $\Delta \in \poly \setminus \varW$, where $\varW$ is a proper algebraic variety with Lebesgue measure zero.
		Further, since the pair~$(A,B)$ is reachable and $\Delta$ is restricted to the structure of $A$ by $\poly$, from \cite[Proposition 2]{shields1976structural}, the pair~$(A+\Delta,B)$ is also reachable for $\Delta \in \poly \setminus \varC$, where $\varC$ is a proper algebraic variety with Lebesgue measure zero.
		Therefore, since a finite union of proper algebraic varieties is a proper algebraic variety, %~\cite[Proposition 1.1]{hartshorne2013algebraic}, 
		we have that the system~$A+\Delta$ will be reachable and will have a unique solution~$W \succ 0$ to~\eqref{eq:lyap_delta} for any $\Delta \in \poly \setminus \varA$, where $\varA \coloneqq \varW \cup \varC$ is a proper algebraic variety with zero Lebesgue measure. Thus, applying Lemma~\ref{thm:lyap_stability_schur}, we have that $A+\Delta$  	
		will be stable for all $\Delta \in \poly \setminus \varA$.
		%, since $\mu(\varA) = \mu(\varW \cup \varC) = \mu(\varW) + \mu(\varC) + \mu(\varW \cap \varC) = 0$.
	\end{proof}	
	Therefore, seeking a tractable computational strategy for~$\Pone$, we consider constraint~\eqref{eq:schur_const} to be implicitly satisfied by all points satisfying
	\eqref{eq:spec} and
	\eqref{eq:lyap_delta} which do not lie in~$\varA$.
	Consequently, if the solution to $\Pone$, as determined by specific constraint sets~$\spec$ and $\poly$, is such that $\Delta \in \varA$, then, we declare $\Pone$ to be infeasible for the parameters defining those sets.  The same considerations apply to $\Ptwo$.
}

%----------------------------------------------------
\subsection{Discrete-time Lyapunov Equation as a Rank Condition}

Notice that, for both problems~\ref{pbm:pone} and \ref{pbm:ptwo}, the discrete-time Lyapunov constraint~\eqref{eq:lyap_delta} induces double and triple products between the decision matrices $\Delta$ and $W$. To address this issue, we first show that \eqref{eq:lyap_delta} can be alternatively satisfied by the solution of a lifted bilinear matrix equation (BME).
Then, we approximate the solution of the resulting BME-constrained problem using a sequence of convex problems.
%Next, we describe each of these steps in detail. 
We begin by lifting the constraint in \eqref{eq:lyap_delta} into a BME using the following lemma.

\begin{lemma}
	\label{lem:disc_lyap}
	The discrete-time Lyapunov equation~\eqref{eq:lyap_delta}
	is satisfied by $W$ and $\Delta$ when the following BME is satisfied by the variables~$W \in \Spp^n$, $H \in \R^{n \times n}$, and $\Delta \in \R^{n \times n}$:
	\begin{align}
	\label{eq:bme_matrix}
	M(W,H)N(\Delta)=Q,
	\end{align}
	where \small
	\begin{align*}
	M(W,H) \!\coloneqq\! 
	\begin{bmatrix}
	H^{\tp} & -W\\
	\!-W & H
	\end{bmatrix}\!,
	N(\Delta) \!\coloneqq\! 
	\begin{bmatrix}
	\!(A +\Delta)^{\tp}\\
	I_{n}
	\end{bmatrix}\!,
	Q \!\coloneqq\!
	\begin{bmatrix}
	-BB^{\tp}\\
	0
	\end{bmatrix}\!\!.
	\end{align*}
\end{lemma}
\begin{proof}
	The equation in \eqref{eq:bme_matrix} is equivalent to the following system of matrix equations:
	\begin{subnumcases}{}
	(A + \Delta)H-W& \!\!\!\!\!\!\!\!\!\!$=-BB^{\tp},$\label{eq:bilinear_line1}\\
	H-W(A + \Delta)^{\tp}& \!\!\!\!\!\!\!\!\!\!\!\! $=0.$\label{eq:bilinear_line2}
	\end{subnumcases}
	From \eqref{eq:bilinear_line2}, we have that $H=W(A + \Delta)^\tp$. Substituting this $H$ in  \eqref{eq:bilinear_line1}, we obtain the Lyapunov equation in~\eqref{eq:lyap_delta}, as desired.
\end{proof}

We now rewrite the BME in \eqref{eq:bme_matrix} as an equivalent rank constraint over a matrix with a specific block structure, as stated in the next theorem.

\begin{theorem}[Rank condition for Lyapunov equation]
	\label{thm:rank_lyap}
	Let $\Z(W,H,\Delta) \in \R^{4n \times 3n}$ be the structured matrix  defined as
	\begin{align} \notag
	\Z(W,H,\Delta) &\coloneqq 
	\begin{bmatrix}
	Z_{11} & Z_{12} \\
	Z_{21} & Z_{22}
	\end{bmatrix}\coloneqq
	\begin{bmatrix}
	I_{2n} & N(\Delta) \\
	M(W,H) & Q
	\end{bmatrix}\\ 
	&= \begin{bmatrix}
	I_{n} & 0 & (A + \Delta)^\tp\\
	0 & I_n & I_n \\
	H^\tp & -W & -BB^\tp \\
	-W & H & 0
	\end{bmatrix}\!\!.
	\label{eq:Zmatrix}
	\end{align}
	%where $M \equiv M(W,H)$, $N\equiv N(\Delta)$, and $Q$ are defined as in~\eqref{eq:bme_matrix}.
	If $\rank[\Z(W^\star,H^\star,\Delta^\star)] = 2n$, then~$W^\star$ and $\Delta^\star$ satisfy the discrete-time Lyapunov equation in~\eqref{eq:lyap_delta}.
\end{theorem}

\begin{proof}
	Consider the Schur complement of $Z_{11}$ in $Z \equiv \Z(W^\star,H^\star,\Delta^\star)$, given by $Z/Z_{11} = Z_{22} - Z_{21}Z_{11}^{-1} Z_{12}$. From \eqref{eq:Zmatrix}, we have that $Z/Z_{11}  = Q - M^\star N^\star$, where $M^\star \equiv M(W^\star,H^\star)$ and $N^\star\equiv N(\Delta^\star)$. 
	According to Guttman's rank additivity formula~\cite{zhang2006schur}, the following holds:
	\begin{align}
	\label{eq:guttman}
	\rank[Z] =  \rank[Z_{11}] + \rank[Z/Z_{11}].
	\end{align}  
	Since $\rank(Z_{11}) = 2n$, we have that $\rank(Z)=2n$ if and only if $\rank[Z/Z_{11}] = 0 = \rank[Q -M^\star N^\star]$, or equivalently, $Q = M^\star N^\star$. Thus, by Lemma~\ref{lem:disc_lyap}, it follows that $W^\star$ and $\Delta^\star$ satisfy the discrete-time Lyapunov equation in~\eqref{eq:lyap_delta}.	
\end{proof}

Equipped with the above result, we can replace the constraint in~\eqref{eq:lyap_delta} by the rank constraint $\rank[\Z(W,H,\Delta)]=2n$ in both problems~\ref{pbm:pone} and \ref{pbm:ptwo}. Importantly, notice that the blocks of $\Z(W,H,\Delta)$ depend affinely on the problem decision matrices $W$ and $\Delta$.
Next, we show that this reformulation can be approached using a sequence of convex programs.

%--------------------------------------------------------------------
\subsection{Design for Reachability via Sequential Optimization}

As introduced in Theorem~\ref{thm:rank_lyap}, a solution~$(W^\star,\Delta^\star)$ to~\eqref{eq:disc_dyn_pert} will be obtained when the rank of $\Z(W^\star,H^\star,\Delta^\star)$ equals $2n$. To achieve this condition, one would in principle seek to minimize the rank of $\Z(W,H,\Delta)$, which is a non-convex and discontinuous function. Alternatively, problems having the rank as an objective function have been approached by considering the nuclear norm (i.e., the sum of a matrix's  singular values) as a relaxation\cite{recht2010guaranteed}. Further, from Theorem~\ref{thm:rank_lyap}, we have a-priori information on the specific optimal value (equal to~$2n$) for the rank of $Z$. In this case, alternative functions related to the nuclear norm have been shown to produce better approximations to the rank function~\cite{hu2012fast}. In particular, the \emph{truncated nuclear norm} function, defined next, uses the rank as an index restricting the number of (ordered) singular values considered in its computation.

\begin{definition}[Truncated nuclear norm function]
	\label{def:tnn}
	The truncated nuclear norm function (TNN) of a matrix~$X \in \R^{m\times n}$ with respect to an integer parameter~$r$ satisfying $r < \min\{m,n\}$ is defined as
	\begin{align*}\small
	\tnn{X}{r} \coloneqq \sum _{i=r+1}^{\min\{m,n\}} \sigma_i(X), \vspace{-3mm}  
	\end{align*}
	where $\sigma_i$ takes values over the set of singular values of $X$ sorted in descending order.
	\label{def:trunc_nuc_norm}
\end{definition}

Using this definition, we can re-state the conditions in Theorem~\ref{thm:rank_lyap} in terms of the TNN, as described below.

\begin{corollary}[TNN sufficient condition for Lyapunov equation]
	\label{cor:suff_lyap}
	If the tuple~$(W^\star \in \Spp,H^\star \in \R^{n \times n},\Delta^\star \in \R^{n \times n})$ satisfies $\tnn{\Z(W^\star,H^\star,\Delta^\star)}{2n} = 0$, then $(W^\star,\Delta^\star)$ satisfies the discrete-time Lyapunov equation~\eqref{eq:lyap_delta}.
	\label{thm:tnn_lyap}
\end{corollary}
\begin{proof}
	The value $\tnn{\Z(W^\star,H^\star,\Delta^\star)}{2n} = 0$ implies
	$\sigma_i = 0$ for $i=2n+1,\ldots,
	%4n$.
	\resub{3n}$.
	 This, in turn, implies that
	$\rank[\Z(W^\star,H^\star,\Delta^\star] = 2n$ in \eqref{eq:Zmatrix}, and subsequently \eqref{eq:lyap_delta} is satisfied by invoking Theorem~\ref{thm:rank_lyap}.
\end{proof}

The next lemma establishes a useful fact associated with Definition~\ref{def:tnn}.

\begin{lemma}[TNN via Von Neumann's inequality \cite{hu2012fast}]
	\label{lem:tnn_von}
	Let $\kfn{X}{r} \coloneqq \sum _{i=1}^{r} \sigma_i(X)$ denote the Ky Fan norm of a matrix~$X \in \R^{m\times n}$ with respect to an integer~$r \leq \min\{m,n\}$.
	%,	where $\sigma_i$ are the singular values of $X$ in ascending order. 
	Then, 
	the TNN can be written as
	\begin{align*}
	\tnn{X}{r} = \| X \|_\ast - \kfn{X}{r},
	\end{align*}	
	which is a \emph{difference-of-convex} function of $X$. Moreover, the TNN is equivalently given by
	\begin{align}
	\label{eq:sup_tnn}
	\tnn{X}{r} = \| X \|_\ast -
	\underset{\substack{LL^\tp = I_r  \\ RR^\tp = I_r}}{\sup}\tr\{LXR^\tp\},
	\end{align}
	for $L \in \R^{r \times m}$ and $R \in \R^{r \times n}$.
\end{lemma}	
\begin{proof}
	We have $\| X \|_\ast - \kfn{X}{r} = \sum_{i=1}^{\min\{m,n\}}\sigma_i(X) - \sum_{i=1}^{r}\sigma_i(X) = \sum _{i=r+1}^{\min\{m,n\}} \sigma_i(X) = \tnn{X}{r}$. This form is clearly a difference of convex functions, since it is a difference between the nuclear and Ky Fan norms of $X$. 
	Equation~\eqref{eq:sup_tnn} is proved by observing the equivalence of $\kfn{X}{r}$ with 
	${\sup_{LL^\tp = I_r ,  RR^\tp = I_r}}\tr\{LXR^\tp\}$,
	as established by Lemma~\ref{lem:von_neumann} in the Appendix. The supremum term is defined over a family of affine functions parameterized by the matrices $L$ and $R$; hence, it is convex.
\end{proof}

%\begin{remark}
%From Lemma~\ref{lem:tnn_von}, we see that the truncated nuclear norm (function) is not strictly a norm, since it does not satisfy the triangle inequality. Notwithstanding, we adopt such nomenclature for compatibility with prior literature~\cite{hu2012fast}.
%\end{remark}

Using Corollary~\ref{cor:suff_lyap}, we can reformulate~\ref{pbm:pone} by seeking to minimize $\tnn{\Z(W,H,\Delta)}{2n}$ subject to the reachability requirements in \eqref{eq:spec} and structural constraints in \eqref{eq:poly}. Using Lemma~\ref{lem:tnn_von}, a solution to~\ref{pbm:pone} can be found by solving the following problem. 

%\label{pbm:ponesub}
\begin{pbm}{$\Ponedc$}[Difference-of-norms problem]\label{pbm:ponedc}
\begin{align*}
\underset{W,H,\Delta}{\minimize} &\quad \| \Z(W,H,\Delta) \|_\ast -\!\!\!\!
\underset{\substack{LL^\tp = I_{2n}\\ RR^\tp = I_{2n}}}{\sup}
\!\!\!\tr\{L\,\Z(W,H,\Delta)\,R^\tp\}\\
\subjecto&\quad W \in \spec, \quad \Delta \in \poly.
\end{align*}
\end{pbm}
\resub{
	% R1-4 [nonfluent]
	%% RESUB
	%Although \ref{pbm:ponedc} is non-convex, since it seeks to minimize the difference of two norms (i.e., the nuclear norm and the Ky Fan norm), it can be seen from Corollary~\ref{cor:suff_lyap} that the global optima of \ref{pbm:ponedc} satisfy $\tnn{\Z(W,H,\Delta)}{2n}=0$ when \ref{pbm:pone} is feasible.
	As established in Theorem~\ref{thm:stab_lyap_generic}, a solution to~\ref{pbm:ponedc} will fulfill the stability constraint in~\eqref{eq:schur_const} \almostsurely{}. 
	Further, despite its non-convexity, \ref{pbm:ponedc} has a known global optimal value when \ref{pbm:pone} is feasible. From Corollary~\ref{cor:suff_lyap}, this optimal value is equal to~$\tnn{\Z(W,H,\Delta)}{2n}=0$.
}

Next, taking inspiration from related problems in the literature~\cite{hu2012fast}, we employ a specific strategy consisting of solving a sequence of convex problems.
More specifically, a convex relaxation of \ref{pbm:ponedc} is obtained by replacing the supremum over parameters $L$ and $R$ in~\eqref{eq:sup_tnn} by fixed values~$\fix L$ and $\fix R$, respectively, as formalized next.

\begin{pbm}{$\Ponesub$}[Convex sub-problem for \ref{pbm:ponedc}]
	\label{pbm:ponesub}
	For fixed $\fix L \in \R^{2n \times 4n}$ and $\fix R \in \R^{2n \times 3n}$, we define the convex problem~$\mathcal{C}( \fix L, \fix R ; \theta)$ as
	\begin{align*}
	\underset{W,H,\Delta}{\minimize} &\quad\| \Z(W,H,\Delta) \|_\ast -	
	\tr\{ \fix  L\,\Z(W,H,\Delta)\, \fix  R^\tp\}\\
	%\subjecto &\quad g(W;\omega) \succeq 0,\quad h(\Delta;\theta) \geq 0.
	\subjecto&\quad W \in \spec, \quad \Delta \in \poly.
	\end{align*}
\end{pbm}

Subsequently, using Von Neumann's trace inequality in Lemma~\ref{lem:von_neumann}, a sequence of convex problems can be defined by iteratively solving \ref{pbm:ponesub} according to the following rule: At each iteration $k$, the parameters~$L^{(k)}$ and $R^{(k)}$ are fixed, and convex sub-problem $\mathcal{C}(L^{(k)}, R^{(k)} ; \theta)$ is solved. Then, the left- and right-singular vectors of the current solution~$\Z^{(k)}(W,H,\Delta) = \argmin_{W,H,\Delta} \mathcal{C}(L^{(k)}, R^{(k)} ; \theta)$ are used, respectively, to update parameters~$L^{(k+1)}$ and $R^{(k+1)}$ for the next iteration. Such procedure, summarized in Algorithm~\ref{alg:alg1}, generates a monotonically convergent sequence of objective function values, as shown in the next theorem.

\newcommand{\StepA}{\textsc{Step~A}}
\newcommand{\StepB}{\textsc{Step~B}}
\begin{algorithm}
	\small
	\begin{algorithmic}[1]
		%\Require
		\Statex \textbf{Inputs:}
		\Statex \qquad reachability parameters $\theta$, tolerance $\epsilon_\eta$
		\Statex \qquad initial value  $Z^{(0)} \leftarrow \Z(W^{(0)},H^{(0)},\Delta^{(0)})$ \vspace{1mm}
		\State $k\leftarrow 0$
		\While{$\tnn{Z^{(k)}}{2n}  \geq \epsilon_\eta $}		
		\Statex \quad \StepA:
		\State $( U^{(k)}, \Sigma^{(k)}, V^{(k)} ) \leftarrow \svd\{Z^{(k)}\}$
		\State $L^{(k)} \leftarrow [u_1^{(k)}|\ldots| u_{2n}^{(k)}]^\tp$, $R^{(k)} \leftarrow [v_1^{(k)}|\ldots|v_{2n}^{(k)}]^\tp$		
		\Statex \quad \StepB:
		\State $(W^{(k+1)},H^{(k+1)},\Delta^{(k+1)}) \!\leftarrow\!\argmin \mathcal{C}(L^{(k)},R^{(k)}; \theta)$
		\State $Z^{(k+1)} \leftarrow \Z(W^{(k+1)},H^{(k+1)},\Delta^{(k+1)})$
		\State $k \leftarrow k+1$
		\EndWhile \vspace{1mm}
		%\Statex \textbf{Outputs:} 
		%\Statex \qquad $ W^{(k)} , \Delta^{(k)}$
	\end{algorithmic}
	\caption{Sequential Convex Program for \ref{pbm:ponedc}}
	\label{alg:alg1}
\end{algorithm}
	
\begin{theorem}[Convergence of Algorithm~\ref{alg:alg1}]
	\label{thm:conv}
	Let $\alpha_k \coloneqq \tnn{\Z(W^{(k)},H^{(k)},\Delta^{(k)})}{2n}$. 
	Then, the sequence $\{\alpha_{k}\}$ generated by $(W^{(k)},H^{(k)},\Delta^{(k)}) = \argmin \,\mathcal C(L^{(k)},R^{(k)}; \theta)$, according to Algorithm~\ref{alg:alg1}, is monotonically non-increasing. 
\end{theorem}
\begin{proof}
	We assume that the sets $\poly$ and $\spec$
	are non-empty, i.e., there exists at least one feasible solution~$(W^{(0)},H^{(0)},\Delta^{(0)})$ to the relaxed problem $\mathcal{C}(L^{(0)},R^{(0)}; \theta)$. 
	\resub{For example, for the worst-case minimum energy design, a feasible solution can be constructed by letting any $\Delta^{(0)} \in \poly$, $W^{(0)}=\tilde \lambda I_n$, and  $H^{(0)}=  W^{(0)}(A + \Delta)^\tp$}.
	Because \StepA{} (in Algorithm~\ref{alg:alg1}) does not affect feasibility of the initial feasible solution~$(W^{(0)},H^{(0)},\Delta^{(0)})$, this solution will remain feasible for~\StepB,  which will also retain feasibility, by construction. Therefore, a solution $(W^{(k)},H^{(k)},\Delta^{(k)})$ will remain feasible at any iteration~$k$.
	Let $\phi(Z,L,R) \coloneqq \| \Z(W,H,\Delta)\|_\ast - \tr\{ L\,\Z(W,H,\Delta)\,R^\tp\}$ be the value of the objective function of $\mathcal C(L,R;\theta)$ evaluated at $Z$, for $Z \equiv \Z(W,H,\Delta)$.
	We now analyze the behavior of the objective function at any iteration~$k$. 
	Denote by $p_A^{(k)} \coloneqq \phi(Z^{(k)},L^{(k)},R^{(k)})$ the objective function value returned after execution of \StepA{} in Algorithm~\ref{alg:alg1}. Likewise, denote by $p_B^{(k)} \coloneqq \phi(Z^{(k+1)},L^{(k)},R^{(k)})$ the objective function value returned after execution of \StepB. Because \StepB{} involves the solution of a (feasible) convex optimization problem, we have $p_B^{(k)} \leq p_A^{(k)}$. Further, by invoking Lemma~\ref{lem:tnn_von}, we have that $p_A^{(k+1)} \leq p_B^{(k)}$. Therefore, we have $p_A^{(k+1)} \leq p_A^{(k)}$ for any $k$, and $\alpha_k = p_A^{(k)}$. Thus, for any $\epsilon_\eta > 0$, there exists an iteration number~$k$ such that $|\alpha_{k+1} - \alpha_{k}| \leq \epsilon_\eta$, and the sequence $\{\alpha_k\}$ is monotonically non-increasing.
\end{proof}

\newcommand{\stat}{\infty}

\iffalse
\begin{remark}
	The results established in Theorem~\ref{thm:conv} guarantee that Algorithm~\ref{alg:alg1} will converge to a stationary value in terms of $\tnn{\Z(W^{(k)},H^{(k)},\Delta^{(k)})}{2n}$.	
However, because of the non-convexity of $\Ponedc$, such a stationarity value does not need to correspond to its global optimal value~$\tnn{\Z(W^{(k)},H^{(k)},\Delta^{(k)})}{2n}=0$, attainable when $\Pone$ is feasible.
\end{remark}
\fi

%---------------------------------------------------------------------------
\subsection{Design for Reachability with Structural Penalties}
\label{sec:penalty}

\resub{
We now build on the results obtained for the feasibility problem~\ref{pbm:pone} to address the more challenging problem~\ref{pbm:ptwo}, which seeks to penalize large magnitudes in the entries of~$\Delta$}. First, we observe that using the definition of the truncated nuclear norm introduced in the previous section, \ref{pbm:ptwo} can be approximated by solving the following problem \resub{for increasing values of the positive weight $\gamma$}.
\begin{pbm}{$\Ptwodc$}[Penalized difference-of-norms problem]
	\label{pbm:ptwodc}
	For $\gamma$ a positive scalar, a relaxation of  \ref{pbm:ptwo} can be written as
	\begin{align*}
	\underset{W,H,\Delta}{\minimize} &	\qquad \tnn{\Z(W,H,\Delta)}{2n} 	+ \gamma \|\Delta \|_1 \\[1mm]
	\subjecto&\quad W \in \spec, \quad \Delta \in \poly
	\end{align*}
	\begin{align*}
	\;\;=\underset{W,H,\Delta}{\minimize} & \qquad \| \Z(W,H,\Delta) \|_\ast + \gamma \|\Delta \|_1 \\[-2mm]
	&\qquad -\!\!\!\! \underset{\substack{LL^\tp = I_{2n}\\ RR^\tp = I_{2n}}}{\sup} \tr\{L\,\Z(W,H,\Delta)\,R^\tp\}\\[3mm]
	\subjecto&\quad W \in \spec, \quad \Delta \in \poly,
	\end{align*}
\end{pbm}
\resub{
	where we have removed the explicit stability constraint \eqref{eq:schur_const} based on the results presented in Theorem~\ref{thm:stab_lyap_generic}.%
}  
Besides using a relaxation strategy similar to the one previously used for \ref{pbm:ponedc} (i.e., replacing the supremum operator with fixed values for $L$ and $R$), we associate with \ref{pbm:ptwodc} the following convex sub-problem.
\begin{pbm}{$\Ptwosub$}[Convex sub-problem for \ref{pbm:ptwodc}]
	\label{pbm:ptwosub}
	For $\gamma >0 $ with fixed $\fix L \in \R^{2n \times m}$ and $\fix R \in \R^{2n \times n}$, we define the convex sub-problem~$\mathcal{C}_{\gamma}(\fix L,  \fix R; \theta)$ as
	\normalsize
	\begin{align*}
	\underset{W,H,\Delta}{\minimize} &\quad\!\!\| \Z(W,H,\Delta) \|_\ast \!-\!
	\tr\{ \fix L\,\Z(W,H,\Delta)\, \!\fix R^\tp\!\}
	\!+ \gamma \| \Delta \|_1 \\
	\subjecto&\quad W \in \spec, \quad \Delta \in \poly.
	\end{align*}
\end{pbm}
\normalsize

%Note that we now have competing objectives, balanced by the weight~$\gamma$.
\resub{
	% R1-4 [nonfluent]
	Note that $\Ptwodc$ presents two competing objectives with relative importance balanced by the weight~$\gamma$. On one hand, we have the truncated nuclear norm term, associated with the residual of the Lyapunov equation~\eqref{eq:lyap_delta}. On the other hand, we have the 1-norm penalty aiming to promote sparsity on the design variable~$\Delta$.%
}
\resub{	
	% R1-4 [nonfluent]
	%% RESUB An important consequence is that the $1$-norm penalty term might act to prevent convergence (in terms of the truncated nuclear norm in $\Ptwodc$) of a sequential strategy similar to the one applied for $\Ponedc$.
	As a result, a sequential optimization strategy similar to the one applied for $\Ponedc$ can introduce an unwanted side-effect: depending on the magnitude of $\gamma$, convergence in terms of the truncated nuclear norm is not guaranteed.%
}
More specifically, while the overall cost of $\Ptwodc$ can be still assured to be monotonically non-increasing (using similar arguments from Theorem~\ref{thm:conv}), higher values of $\gamma$ might promote iterations where a decrease in the overall objective function (including the penalty term $\gamma \| \Delta \|_1$) will be obtained at the expense of an increase in the term associated with the truncated nuclear norm $\| \Z(W,H,\Delta) \|_\ast -	
\tr\{ \fix L\,\Z(W,H,\Delta)\, \fix R^\tp\}$.

\resub{
	% R2-6b [regularization path]
	% Also I could not fully understand the details of ``regularization path" argument before/about Algorithm~2.
	%% RESUB
	%To control this effect, we propose the following iterative procedure. We first solve \ref{pbm:ptwodc} using Algorithm~\ref{alg:alg1} with $\gamma = 0$, thereby  taking benefit from the convergence characteristic in the term associated with $\tnn{\Z(W,H,\Delta)}{2n}$. If Algorithm~\ref{alg:alg1} produces a solution $\Z(\sta W, \sta H,\sta \Delta)$ with $\tnn{\Z(\sta W, \sta H,\sta \Delta)}{2n}=0$, we use $L^{(0)} = [u_1^{(0)},\ldots, u_{2n}^{(0)}]^\tp$,  and $R^{(0)} = [v_1^{(0)},\ldots,v_{2n}^{(0)}]^\tp$ from the $\svd \{ Z^{(0)}\}$ with $Z^{(0)}= \Z(\sta W, \sta H,\sta \Delta)$ as initial parameters. We then attempt to solve \ref{pbm:ptwodc} for increasing values of $\gamma$ until a stopping criterion is met -- a strategy commonly referred to as \emph{regularization path}~\cite{giesen2012regularization,blomberg2014approximate}. In particular, we choose to stop Algorithm~\ref{alg:alg2} at iteration~$k^\star$ if $\tnn{Z^{(k)}}{2n} \geq  \tnn{Z^{(k-1)}}{2n}$ holds for $K> 1$ (a parameter of choice) successive iterations~$k$ preceding $k^\star$. More formally, we define the function $\mathrm{proceed}_K(Z^{(\resub{\min\{0,k-K+1\}})}, \ldots, Z^{(k)})$, which returns \textsc{false} if $\tnn{Z^{(k)}}{2n} \geq  \tnn{Z^{(k-1)}}{2n}$ for $k-K+2,\ldots,k$ when $k\geq K$, and \textsc{true} otherwise. The proposed procedure is summarized in Algorithm~\ref{alg:alg2}.
	To control this effect, we propose an iterative procedure that seeks an approximation for the largest value of $\gamma$ for which $\Ptwodc$ can be solved. The proposed procedure begins by
	solving \ref{pbm:ptwodc}$(\gamma)$ with $\gamma = 0$. In this configuration, \ref{pbm:ptwodc}$(\gamma)$ is equivalent to the unpenalized problem~\ref{pbm:ponedc}. Therefore,   
	Algorithm~\ref{alg:alg1} can be applied to achieve convergence as established in Theorem~\ref{thm:conv}. 
	Then, we attempt to solve \ref{pbm:ptwodc}$(\gamma)$ for increasing values of~$\gamma$, using the solution of the current problem as an initialization for the next problem, until a stopping criterion is met.
	This type of strategy is commonly referred to as \emph{regularization path}, and has been applied to control problems, for instance, in~\cite{giesen2012regularization,blomberg2014approximate}.

	Formally, we consider a sequence~$\{\gamma_t\}_{t=1}^N$ of increasing positive weights, and begin by applying Algorithm~\ref{alg:alg1} to solve $\Ptwodc$$(\gamma_0)$ with a preliminary weight~$\gamma_0=0$. If Algorithm~\ref{alg:alg1} fails to produce a feasible solution at convergence, we declare $\Ptwodc$ infeasible.
	Otherwise, if it produces a solution~$\Z(\sta W, \sta H,\sta \Delta)$ with $\tnn{\Z(\sta W, \sta H,\sta \Delta)}{2n} < \epsilon_\eta$, we make $Z^{(0)} \equiv \Z(\sta W, \sta H,\sta \Delta)$ and use $L^{(0)} = [u_1^{(0)},\ldots, u_{2n}^{(0)}]^\tp$ and $R^{(0)} = [v_1^{(0)},\ldots,v_{2n}^{(0)}]^\tp$ from $\svd \{ Z^{(0)}\}$ as initial parameters for $\Ptwodc$$(\gamma_1)$. Then, for each $\gamma_t$, we seek to solve $\Ptwodc$$(\gamma_t)$ by a sequence of convex subproblems~$\{\mathcal{C}_{\gamma_t}(L^{(k)},R^{(k)}; \theta)\}_k$ and evaluate the stopping condition in terms of the inner-loop solution~$Z^{(k)}\equiv \Z(W^{(k)}, H^{(k)},\Delta^{(k)})$ to each $\mathcal{C}_{\gamma_t}(L^{(k)},R^{(k)}; \theta)$, as follows. If $\tnn{Z^{(k)}}{2n} < \epsilon_\eta$, we consider the algorithm to have converged for the current weight~$\gamma_t$, and move on to the next weight in the sequence.
	Otherwise, we choose to stop the sequence if $\tnn{Z^{(k)}}{2n} \geq  \tnn{Z^{(k-1)}}{2n}$ holds for~$K>1$ successive iterations of $\mathcal{C}_{\gamma_t}(L^{(k)},R^{(k)}; \theta)$, where $K$ is a parameter of choice. 
	For this purpose, we define the function $\mathrm{stop}_K(Z^{(\resub{\min\{0,k-K+1\}})}, \ldots, Z^{(k)})$, which returns \textsc{true} if 
	$\tnn{Z^{(k)}}{2n} \geq  \tnn{Z^{(k-1)}}{2n}$ for $k-K+2,\ldots,k$ when $k\geq K$, and \textsc{false} otherwise.
	The proposed procedure is summarized in Algorithm~\ref{alg:alg2}.
}

\begin{algorithm}	\small
	\begin{algorithmic}[1]
		\Statex \textbf{Inputs:}
		\Statex \qquad parameters $\theta$,  tolerance $\epsilon_\eta$, stopping number $K$
		\Statex \qquad penalization weights $\gamma_0=0$ and $\gamma_1, \ldots, \gamma_N$
		\Statex \qquad initial value  $Z^{(0)} \leftarrow \Z(W^{(0)},H^{(0)},\Delta^{(0)})$
		\vspace{1mm}
		\State $k\leftarrow 0$, $t \leftarrow 0$		
		\While{\resub{\textbf{not }$\mathrm{stop}_K(Z^{(\min\{0,k-K+1\})}, \ldots, Z^{(k)})$}}
		\While{$\tnn{Z^{(k)}}{2n}  \geq \epsilon_\eta $}
		\Statex \qquad \StepA:
		\State  \!\!\!\!\!\!$( U^{(k)}, \Sigma^{(k)}, V^{(k)} ) =  \svd\{Z^{(k)}\}$
		\State \!\!\!\!\!\!$L^{(k)} \leftarrow [u_1^{(k)},\ldots, u_{2n}^{(k)}]^\tp$, $R^{(k)} \leftarrow [v_1^{(k)},\ldots,v_{2n}^{(k)}]^\tp$		
		\Statex \qquad \StepB:
		\State \!\!\!\!\!\!$(W^{(k+1)},H^{(k+1)},\Delta^{(k+1)}) \!\leftarrow \!\argmin \mathcal{C}_{\gamma_t}(L^{(k)},R^{(k)}; \theta)$
		\State \!\!\!\!\!\!$Z^{(k+1)} \leftarrow \Z(W^{(k+1)},H^{(k+1)},\Delta^{(k+1)})$
		\State \!\!\!\!\!\!$ k \leftarrow k  + 1$
		\EndWhile
		\State $(W^{(0)},H^{(0)},\Delta^{(0)}) \leftarrow (W^{(k+1)},H^{(k+1)},\Delta^{(k+1)})$
		\State $t\leftarrow t+1$		
		\EndWhile \vspace{1mm}
		%\Statex \textbf{Outputs:} 
		%\Statex  \qquad	$W^{(k)} ,\Delta^{(k)}$
	\end{algorithmic}
	\caption{Regularization Path Algorithm for \ref{pbm:ptwodc}}
	\label{alg:alg2}
\end{algorithm}

%==================================
\section{Computational Experiments}

\newcommand{\impMinLam}{design for worst-case reachabililty}
\newcommand{\impTrInv}{design for average reachability}
\newcommand{\ratio}{\rho}

To illustrate the effectiveness of our proposed approaches, in this section we perform several computational experiments considering both worst-case and average reachability designs. 
In the first set of experiments, we analyze random networks generated by the directed \erdos{} model. The main goal is to verify the convergence of our algorithm for different random system realizations and different reachability objectives. As we will illustrate, our algorithm typically reaches solutions characterized by a very low value (i.e., below a pre-specified tolerance) of the truncated nuclear norm after a relatively small number of iterations.

In the second set of experiments, we examine a networked system with the topology of the \ieeebus{}\cite{christie2000power}. 
\resub{
	%% R2-1 [bulk]
	We take inspiration from~\cite{summers2014submodularity}, which considers the problem of improving transient stability properties of power grids to damp frequency oscillations and prevent rotor angle instability. In this setting, the physical design variables are associated with the placement of high voltage direct current (HVDC) links, which are modeled as ideal AC current sources on the terminal buses~\cite{fuchs2013actuator}. Further, in their problem formulation, the nonlinear swing equations of system are linearized, and the HVDC placements are evaluated using controllability Gramian metrics. Our presentation consists of a simplification of the aforementioned experiment, with the goal of illustrating the effects of sparsity obtained by applying the procedure for design with structural penalties described in Section~\ref{sec:penalty}. Further, as described in our problem statement, we restrict our edge design variables to follow the existing network topology. 
}% 
The code and data generated for both sets of experiments are available in~\cite{github2018netdeco}.

%----------------------------------------
\subsection{\erdos}

\begin{figure*}[h]
	\centering
	\includegraphics
	[height=0.28\textheight]
	{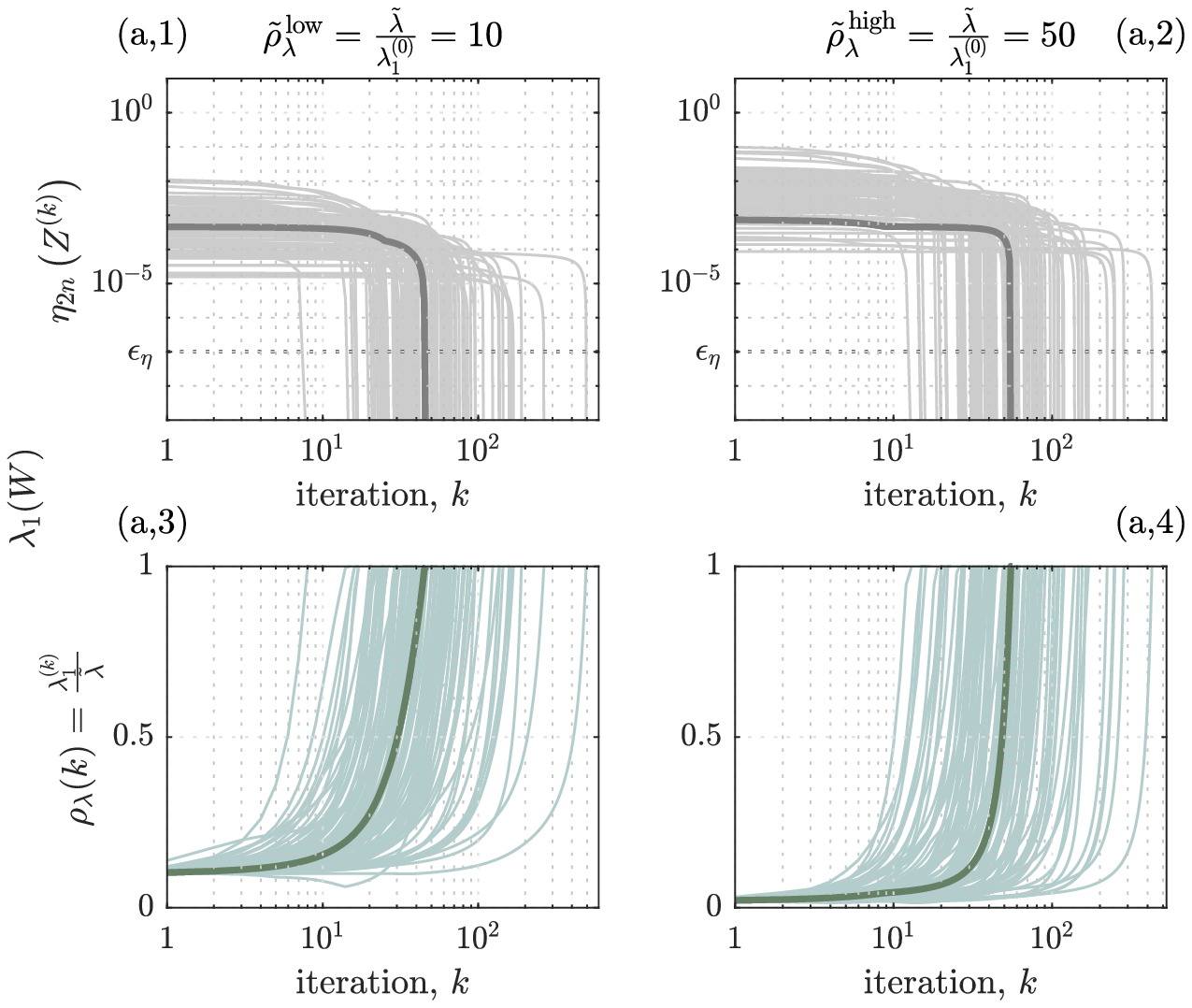}
	\includegraphics
	%[width=1\linewidth]
	[height=0.28\textheight]
	{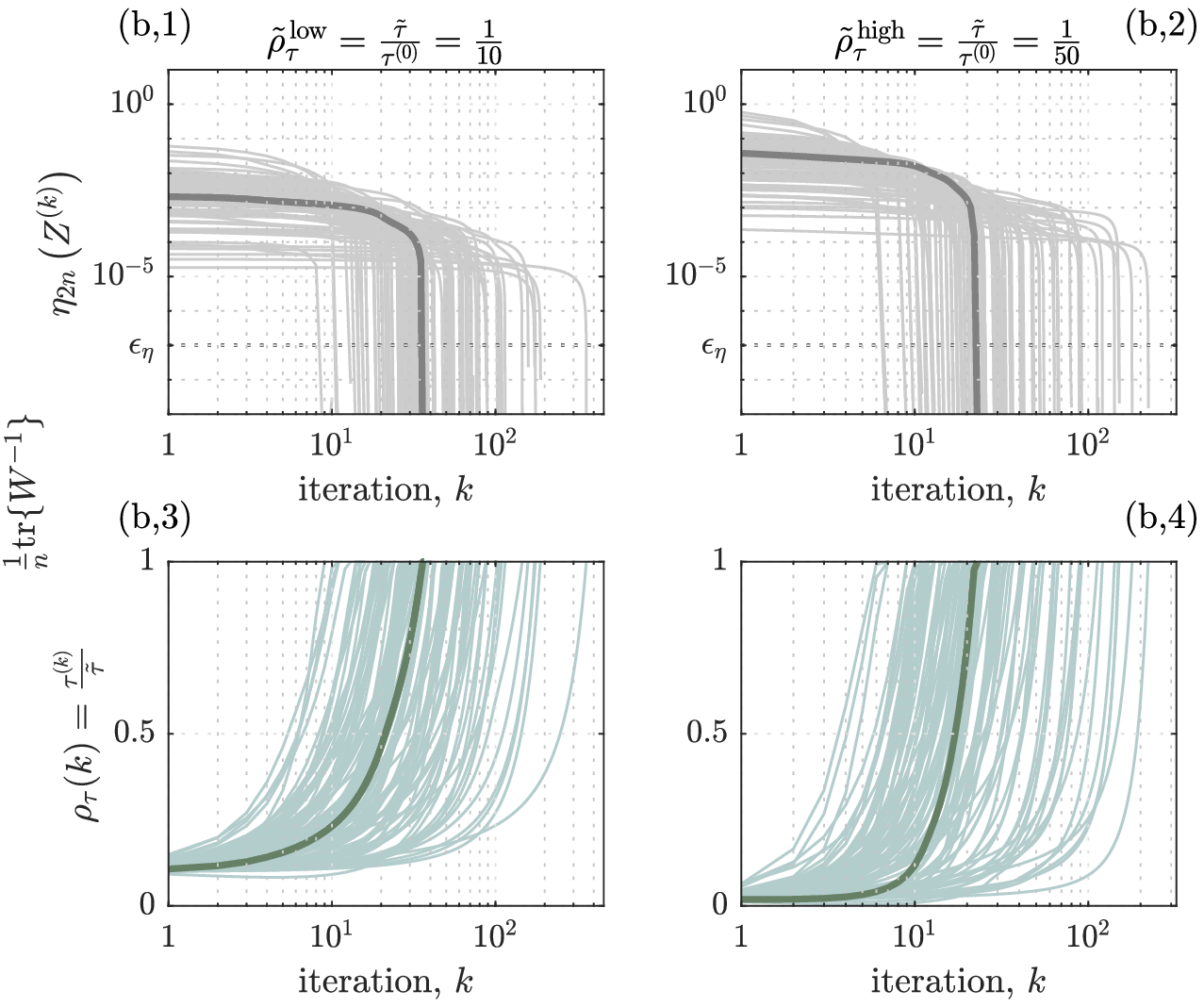}
	\caption{
		Improvement of reachability for \erdos{} systems for $L=100$ random realizations of $(A,B)$ system pairs. In (a,$\cdot$), we consider the \impMinLam{} problem, while in (b,$\cdot$) we present results for the \impTrInv{} problem.
		For the first case, \ressub{panels}~(a,1) and (a,2) present the truncated nuclear norm~$\tnn{Z^{(k)}}{2n}$ as a function of the algorithm iteration~$k$, considering, respectively, low and high target reachability improvement values 
		\ressub{(i.e., $\tilde \ratio_\lambda^{\mathrm{low}}$ and $\tilde  \ratio_\lambda^{\mathrm{high}}$)}.
		Correspondingly, (a,3) and (a,4) display the current-to-target reachability improvement ratios~\ressub{$\ratio_\lambda(k)= \lambda_1(k) / \tar \lambda$} for the same system realizations and low/high improvement targets. A value of \ressub{$\ratio_\lambda(k)\geq 1$} implies the achievement of the desired reachability improvement \ressub{$\lambda_1(k) \geq   \tar \lambda$}. 
		Each thin line is associated with one of the $L=100$ \ressub{random \erdos{}} system realizations. The thicker line is associated with \ressub{the specific} system realization 
		\ressub{whose iteration number when the stopping criterion was met was in the median of the stopping iteration numbers for all system realizations.
			Likewise, panels~(b,1) and (b,2) display the truncated nuclear norm~$\tnn{Z^{(k)}}{2n}$ considering, respectively, low and high reachability improvement target values for the \impTrInv{} problem
			(i.e., $\tilde \ratio_\tau^{\mathrm{low}}$ and $\tilde  \ratio_\tau^{\mathrm{high}}$). Correspondingly, panels~(b,3) and (b,4) display the current-to-target reachability ratios~$\ratio_\tau(k)= \tau(k) / \tar \tau$ for the same system realizations and low/high improvement targets. A value of $\ratio_\tau(k)\geq 1$ implies the achievement of the desired reachability improvement $\tau(k) \leq  \tar \tau$.}		
	}
	\label{fig:results_erdos}
\end{figure*}

We generate $L=\sval{res_erdos}{par.p.l}$ random realizations of directed \erdos{} (ER) systems, with state dimension~\resub{$n=\sval{res_erdos}{par.p.n}$} and input dimension~\resub{$m=\sval{res_erdos}{par.p.m}$}.
Each system~$l=1,\ldots,L$ is defined by a pair $(A^{(l)},B^{(l)})$ that is generated as follows: The sparsity pattern encoded by the set $\{ (i,j) : i,j = 1,\ldots,n; (i,j) \in \G \}$, is obtained by following the ER process until the resulting density of nonzero entries, i.e., $\|A^{(l)}\|_0/n^2$, reaches a value of~\resub{$0.5$}.
The weights of the edges in the network are sampled from a standard uniform distribution, i.e., $[A^{(l)}]_{i,j} \sim \text{uniform}(0,1)$, for all $(i,j) \in \G$, with self-loops being allowed. To assure stability, the entries of each matrix $A^{(l)}$ were simultaneously scaled such that the absolute value of the largest eigenvalue of the matrix was less than one. The entries of the input matrices $B^{(l)} = [b_1^{(l)}| \ldots| b_m^{(l)}]$ were selected to have each column~$b_j$ ($j=1\ldots,m$) defined as a canonical indicator vector $e_{\pi_j(n)}$, where $\pi_j(n)$ denotes the index of the entry equal to $1$ and is obtained as a random permutation of the $1,\ldots,n$ possible indices. Each pair was tested for reachability by assuring that $\rank[\,\mathcal C (A^{(l)},B^{(l)})\,] = n$, where $\mathcal C(A,B) = [\, B\, | \, AB\,| \cdots | \,A^{n-1}B\, ]$.
%
%
%An illustration of a network is presented in Figure~\ref{fig:graph_erdos}.

We consider two types of design problems:
\emph{(i)} \emph{\impMinLam{}}, associated with the minimum eigenvalue~$\lambda_1(W)$,  and
\emph{(ii)} \emph{\impTrInv{}}, associated with~$\tau = \frac{1}{n}\tr\{ W^{-1}\}$.
For each objective, we explore two cases: one with a \emph{low} target improvement value, and one with a \emph{high} target improvement value. For the case of \impMinLam{}, we define the ratio of improvement~$\ratio_\lambda = \tar \lambda_1 /   \lambda_1$ 
and fix target values 
$\tar \ratio_\lambda^{\text{\,low}} = \sval{res_erdos}{par.s.lam_min_gains(1)}$
 and 
$\tar \ratio_\lambda^{\text{\,high}} = \sval{res_erdos}{par.s.lam_min_gains(2)}$.
For the case of \impTrInv{}, we define the ratio of improvement~$\ratio_\tau = \tar  \tau / \tau$ and fix target values 
$\tar \ratio_\tau^{\text{\,low}} = \frac{1}{\sval{res_erdos}{par.s.tr_inv_gains_rec(1)}}$
and 
$\tar \ratio_\tau^{\text{\,high}} =  \frac{1}{\sval{res_erdos}{par.s.tr_inv_gains_rec(2)}}$. The maximum and minimum allowed perturbation magnitudes~$[\Delta]_{i,j}$ were set to $\upsilon_{i,j}=\sval{res_erdos}{par.s.amax}$ and
$\iota_{i,j}=\sval{res_erdos}{par.s.amin}$, respectively, for all $i$ and $j$.
We then observe the evolution of the truncated nuclear norm $\tnn{Z^{(k)}}{2n}$ as a function of the iteration $k$ for each system realization, until a stopping criterion is met. In particular, 
this criterion was set to $\epsilon_\eta = \sval{res_erdos}{par.m.tol_tnn}$, i.e., the algorithm stops when it reaches an iteration $k^\star$ for which $\tnn{Z^{(k^\star)}}{2n} \leq \epsilon_\eta$. 
The results from the execution of the algorithm are presented in Figure~\ref{fig:results_erdos}.
It can be seen that $\tnn{Z^{(k)}}{2n}$ reached the threshold $\epsilon_\eta$ for all cases considered, indicating that the desired reachability improvement, as captured by the constraint $W \in \spec$, was feasible in relation to the structural constraints imposed by $\Delta \in \poly$. Further, the median iteration value $k^\star$ for which such threshold was achieved is below 100 for the four scenarios considered. 
Finally, it can observed that the iteration for which the desired improvement in reachability is achieved typically coincides with the iteration at which the truncated nuclear norm reaches the lowest point.

%--------------------------------------------
\subsection{IEEE Electric Power Network}

\newcommand{\gI}{\gamma_{\text{\,first}}}
\newcommand{\gF}{\gamma_{\text{\,final}}}
\newcommand{\gL}{\gamma_{\text{\,last}}}

We generate a network following the topology of the \ieeebus{}~\cite{christie2000power}, with state dimension~$n=\sval{res_ieeebus}{par.p.n}$ and input dimension~$m=\sval{res_ieeebus}{par.p.m}$. 
%(see illustration in Figure~\ref{fig:ieee-14-bus-network-system}).
%
The maximum and minimum allowable perturbation magnitudes~$[\Delta]_{i,j}$ are set to $\upsilon_{i,j}=\sval{res_ieeebus}{par.s.amax}$ and
$\iota_{i,j}=\sval{res_ieeebus}{par.s.amin}$, respectively, for all $i$ and $j$.
\resub{As a simplification of the experiments presented in~\cite{summers2014submodularity}}, the initial weights of the network were symmetrically associated with the resistance values of the transmission lines, with particular numerical values set to those available in \cite{sousa2018}. 
The resulting matrix $A$ has sparsity pattern and weights as displayed next, with values rounded for compactness.
\begin{align*}
&\;\,A = \\
&\left [ \begin{smallmatrix}
\,\,\cdot\,\, & 0.06 & \,\,\cdot\,\, & \,\,\cdot\,\, & 0.22 & \,\,\cdot\,\, & \,\,\cdot\,\, & \,\,\cdot\,\, & \,\,\cdot\,\, & \,\,\cdot\,\, & \,\,\cdot\,\, & \,\,\cdot\,\, & \,\,\cdot\,\, & \,\,\cdot\,\, \\ 
0.06 & \,\,\cdot\,\, & 0.20 & 0.18 & 0.17 & \,\,\cdot\,\, & \,\,\cdot\,\, & \,\,\cdot\,\, & \,\,\cdot\,\, & \,\,\cdot\,\, & \,\,\cdot\,\, & \,\,\cdot\,\, & \,\,\cdot\,\, & \,\,\cdot\,\, \\ 
\,\,\cdot\,\, & 0.20 & \,\,\cdot\,\, & 0.17 & \,\,\cdot\,\, & \,\,\cdot\,\, & \,\,\cdot\,\, & \,\,\cdot\,\, & \,\,\cdot\,\, & \,\,\cdot\,\, & \,\,\cdot\,\, & \,\,\cdot\,\, & \,\,\cdot\,\, & \,\,\cdot\,\, \\ 
\,\,\cdot\,\, & 0.18 & 0.17 & \,\,\cdot\,\, & 0.04 & \,\,\cdot\,\, & 0.21 & \,\,\cdot\,\, & 0.56 & \,\,\cdot\,\, & \,\,\cdot\,\, & \,\,\cdot\,\, & \,\,\cdot\,\, & \,\,\cdot\,\, \\ 
0.22 & 0.17 & \,\,\cdot\,\, & 0.04 & \,\,\cdot\,\, & 0.25 & \,\,\cdot\,\, & \,\,\cdot\,\, & \,\,\cdot\,\, & \,\,\cdot\,\, & \,\,\cdot\,\, & \,\,\cdot\,\, & \,\,\cdot\,\, & \,\,\cdot\,\, \\ 
\,\,\cdot\,\, & \,\,\cdot\,\, & \,\,\cdot\,\, & \,\,\cdot\,\, & 0.25 & \,\,\cdot\,\, & \,\,\cdot\,\, & \,\,\cdot\,\, & \,\,\cdot\,\, & \,\,\cdot\,\, & 0.20 & 0.26 & 0.13 & \,\,\cdot\,\, \\ 
\,\,\cdot\,\, & \,\,\cdot\,\, & \,\,\cdot\,\, & 0.21 & \,\,\cdot\,\, & \,\,\cdot\,\, & \,\,\cdot\,\, & 0.18 & 0.11 & \,\,\cdot\,\, & \,\,\cdot\,\, & \,\,\cdot\,\, & \,\,\cdot\,\, & \,\,\cdot\,\, \\ 
\,\,\cdot\,\, & \,\,\cdot\,\, & \,\,\cdot\,\, & \,\,\cdot\,\, & \,\,\cdot\,\, & \,\,\cdot\,\, & 0.18 & \,\,\cdot\,\, & \,\,\cdot\,\, & \,\,\cdot\,\, & \,\,\cdot\,\, & \,\,\cdot\,\, & \,\,\cdot\,\, & \,\,\cdot\,\, \\ 
\,\,\cdot\,\, & \,\,\cdot\,\, & \,\,\cdot\,\, & 0.56 & \,\,\cdot\,\, & \,\,\cdot\,\, & 0.11 & \,\,\cdot\,\, & \,\,\cdot\,\, & 0.08 & \,\,\cdot\,\, & \,\,\cdot\,\, & \,\,\cdot\,\, & 0.27 \\ 
\,\,\cdot\,\, & \,\,\cdot\,\, & \,\,\cdot\,\, & \,\,\cdot\,\, & \,\,\cdot\,\, & \,\,\cdot\,\, & \,\,\cdot\,\, & \,\,\cdot\,\, & 0.08 & \,\,\cdot\,\, & 0.19 & \,\,\cdot\,\, & \,\,\cdot\,\, & \,\,\cdot\,\, \\ 
\,\,\cdot\,\, & \,\,\cdot\,\, & \,\,\cdot\,\, & \,\,\cdot\,\, & \,\,\cdot\,\, & 0.20 & \,\,\cdot\,\, & \,\,\cdot\,\, & \,\,\cdot\,\, & 0.19 & \,\,\cdot\,\, & \,\,\cdot\,\, & \,\,\cdot\,\, & \,\,\cdot\,\, \\ 
\,\,\cdot\,\, & \,\,\cdot\,\, & \,\,\cdot\,\, & \,\,\cdot\,\, & \,\,\cdot\,\, & 0.26 & \,\,\cdot\,\, & \,\,\cdot\,\, & \,\,\cdot\,\, & \,\,\cdot\,\, & \,\,\cdot\,\, & \,\,\cdot\,\, & 0.20 & \,\,\cdot\,\, \\ 
\,\,\cdot\,\, & \,\,\cdot\,\, & \,\,\cdot\,\, & \,\,\cdot\,\, & \,\,\cdot\,\, & 0.13 & \,\,\cdot\,\, & \,\,\cdot\,\, & \,\,\cdot\,\, & \,\,\cdot\,\, & \,\,\cdot\,\, & 0.20 & \,\,\cdot\,\, & 0.35 \\ 
\,\,\cdot\,\, & \,\,\cdot\,\, & \,\,\cdot\,\, & \,\,\cdot\,\, & \,\,\cdot\,\, & \,\,\cdot\,\, & \,\,\cdot\,\, & \,\,\cdot\,\, & 0.27 & \,\,\cdot\,\, & \,\,\cdot\,\, & \,\,\cdot\,\, & 0.35 & \,\,\cdot\,\, \\ 
\end{smallmatrix}\right ] 
	
\end{align*}
In the above matrix, the symbol `$\cdot$' denotes an absence of interconnection, corresponding to an entry with numerical value $0$.
In particular, the network represented by $A$ has a total of
$\sval{res_ieeebus}{hst.card_A0}$
edges out of  
$\sval{res_ieeebus}{hst.num_A0}$
possible, resulting in a 
% sparsity level of  $\sval{res_ieeebus}{hst.sparsity_A0}$.
density of $0.204$ nonzero entries.

\begin{figure*}[t]
	\centering
	\includegraphics
	%[width=1\linewidth]
	[height=0.28\textheight]
	{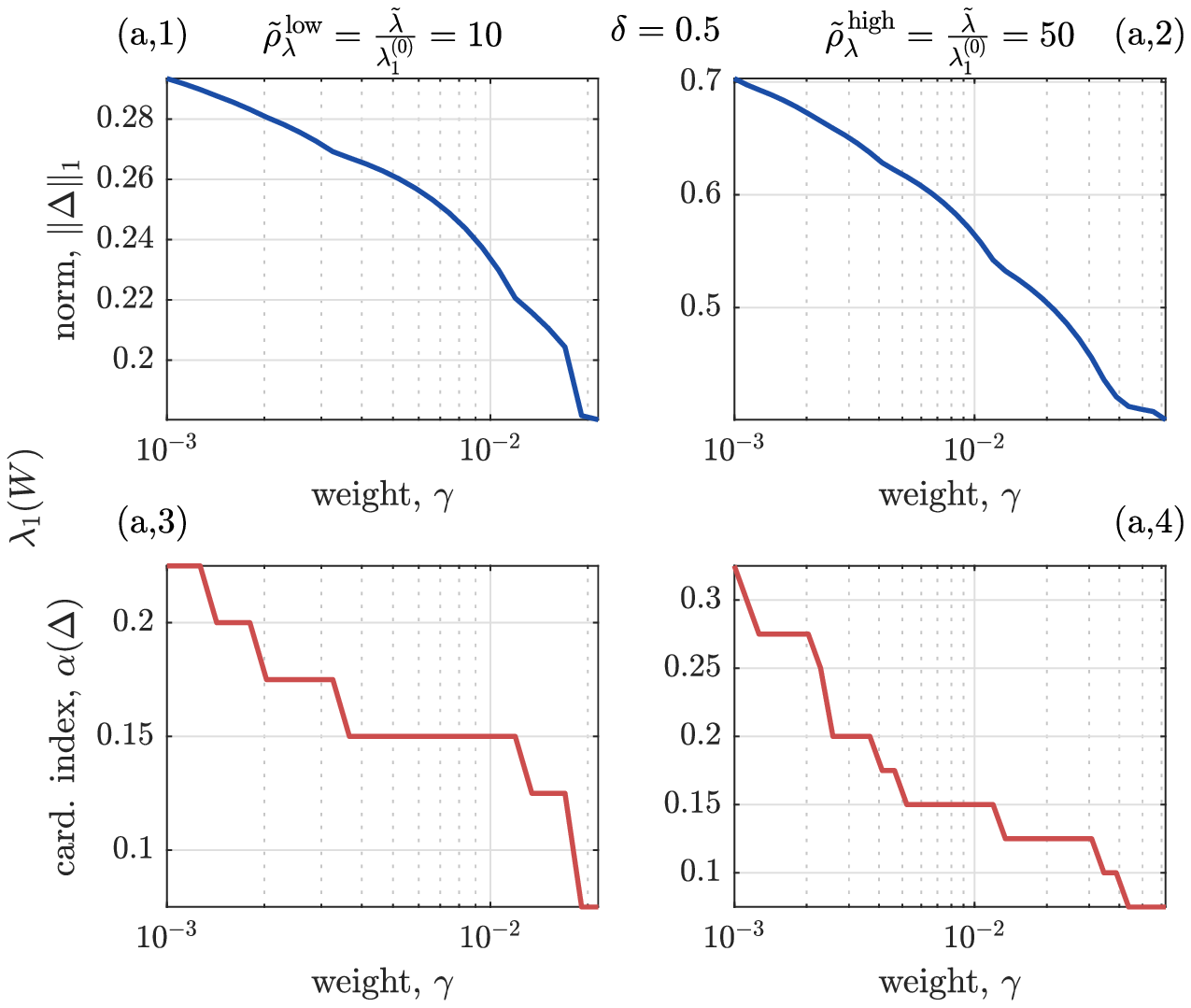}
	\includegraphics
	%[width=1\linewidth]
	[height=0.28\textheight]
	{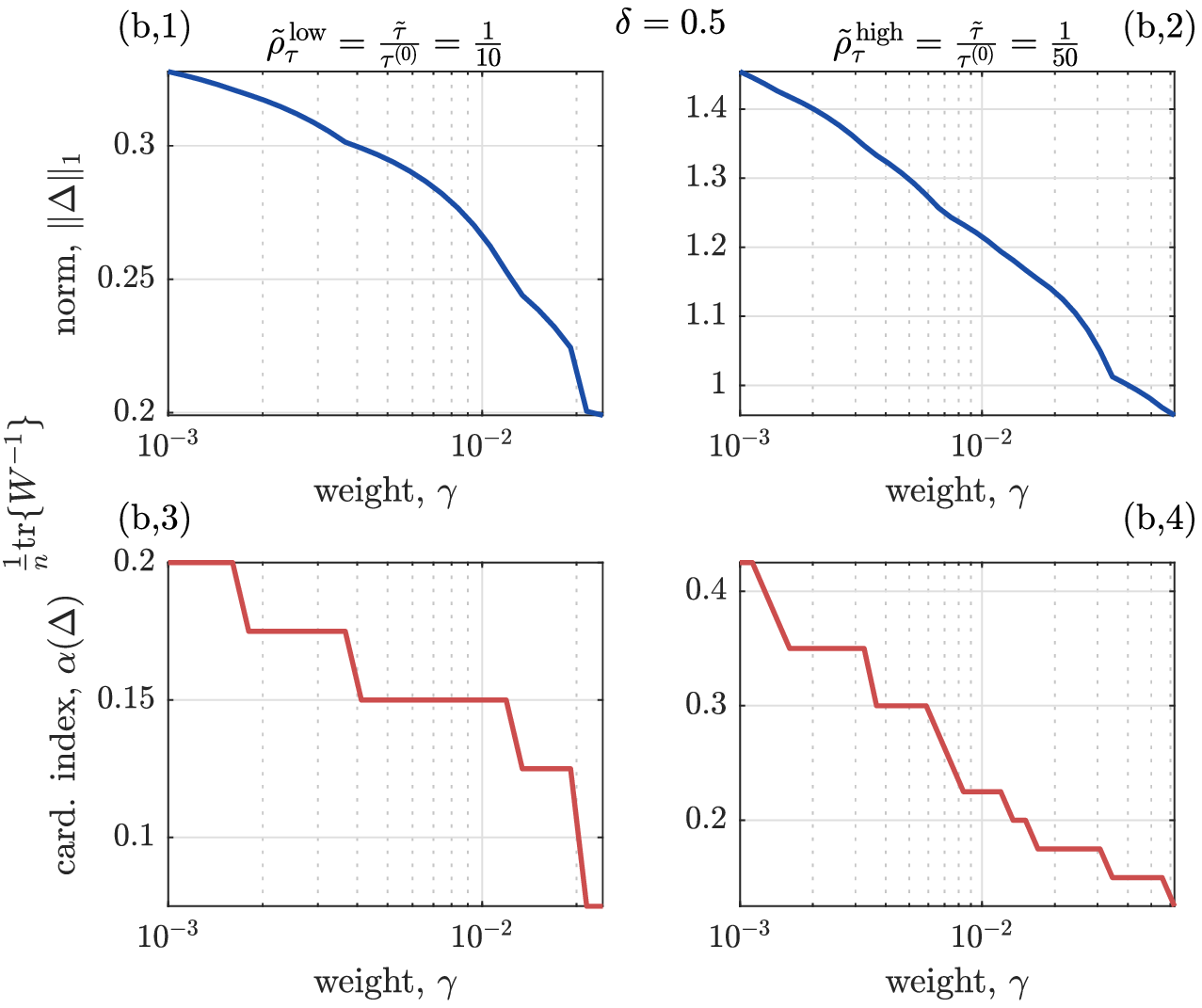}
	\caption{
		Reachability design with induced sparsity for the \ieeebus.
		In (a,$\cdot$), we consider the \impMinLam{} problem, while in (b,$\cdot$) we present results for the \impTrInv{} problem.
		For the first case, (a,1) and (a,2) present the 1-norm of the matrix~$\Delta$ as a function of the penalization weight~$\gamma$, considering low and high target reachability improvement values, respectively. Correspondingly, (a,3) and (a,4) display the cardinality index~$\alpha(\Delta)$ for the same system realizations when low and high improvement targets are considered.
		Likewise, in (b,1) and (b,2)  (resp. (b,3) and (b,4)), we display the 1-norm  (resp.  cardinality index) for low and high reachability improvement target values, when the \impTrInv{} problem is considered.
		\ressub{
			In terms of the simplified power system network analyzed, the decrease in the cardinality index value for increasing values of $\gamma$ seen in panels (a,3), (a,4), (b,3), and (b,4) means that a decreasing number of high voltage direct current (HVDC) links would need to be deployed in order for the system to achieve the desired controllability metrics (i.e., minimal worst-case and average energy required at the control inputs).
			}
		}
	\label{fig:results_ieeebus}
\end{figure*}

In a similar fashion to the previous experiment, we consider two types of design:
\emph{(i)} \impMinLam{}, associated with the minimum eigenvalue~$\lambda_1(W)$,  and
\emph{(ii)} \impTrInv{}, associated with~$\tau = \frac{1}{n}\tr\{ W^{-1}\}$.
For each objective, we explore two cases: one with a low target improvement value, and one with a high target improvement value. For case of \impMinLam{}, we define the ratio of improvement~$\ratio_\lambda = \tar \lambda_1 /  \lambda_1$ 
and set target values 
$\tar \ratio_\lambda^{\text{\,low}} = \sval{res_erdos}{par.s.lam_min_gains(1)}$
and 
$\tar \ratio_\lambda^{\text{\,high}} = \sval{res_erdos}{par.s.lam_min_gains(2)}$.
For the case of \impTrInv{}, we define the ratio of improvement~$\ratio_\tau = \tar \tau / \tau$ and set target values 
$\tar \ratio_\tau^{\text{\,low}} = \frac{1}{\sval{res_erdos}{par.s.tr_inv_gains_rec(1)}}$
and 
$\tar \ratio_\tau^{\text{\,high}} =  \frac{1}{\sval{res_erdos}{par.s.tr_inv_gains_rec(2)}}$.

To evaluate the effect of the sparsity inducing penalty, we define the \emph{cardinality index} $\alpha(\Delta ) \coloneqq \| \Delta \|_0 / \| A \|_0$, which aims at computing the density of nonzero entries of $\Delta$ in terms of the available system entries, as induced by the sparsity pattern of the original system matrix~$A$.
We solve \ref{pbm:ptwodc} using Algorithm~\ref{alg:alg2} for $\sval{res_ieeebus}{par.s.num_gam}$ different values of the penalization parameter $\gamma$, whose logarithm values are set to be uniformly spaced in the pre-specified interval
$\log_{10} \gamma \in [\sval{res_ieeebus}{par.s.min_gam},\sval{res_ieeebus}{par.s.max_gam}]$. 
\ressub{
	In practice, this range just needs to be chosen wide enough such that its
	lower limit allows $\Ptwodc$ to be solved within the prescribed tolerance, and, conversely, its upper limit causes $\Ptwodc$ not to be solved (i.e, the $\mathrm{stop}_K$ function  returns $\textsc{true}$ at some iteration~$k^\star$).
	 %More specifically, the lower limit can be chosen such that penalization term~$\gamma \| \Delta \|$, when evaluated at a solution~$\Delta^\star(\gamma)$ obtained for $\gamma=0$, represents a small fraction of the total objective value. Correspondingly, the upper limit can be chosen such that penalization term corresponds to a high fraction (i.e., within a few orders of magnitude) of the objective value. 
}
In particular, Algorithm~\ref{alg:alg2} is set to stop at iteration $k^\star$ if $\tnn{Z^{(k)}}{2n} \geq  \tnn{Z^{(k-1)}}{2n}$ holds for 
$K =  \sval{res_ieeebus}{par.m.max_no_decrease}$ successive iterations preceding $k^\star$.
The results from the execution of the algorithm are presented in Figure~\ref{fig:results_ieeebus}. We notice the decrease of the penalty term $\| \Delta\|_1$ associated with a decrease in the cardinality index $\alpha(\Delta)$, for all the four cases studied.
		\ressub{
	The total number of iterations (i.e., convex subproblems solved) for the worst-case controllability metric was of $47$ and $61$ for the low and high improvement ratios, respectively. Likewise, the total number of iterations for the average controllability metric was of $49$ and $60$, respectively, for the low and high improvement ratios.
}
Further, for concreteness, we display the specific values of $\Delta$ for the initial and final values of the penalization weight~$\gamma$, considering the scenario where we seek the \impTrInv{} with a high target value of improvement  
$\tar \ratio_\tau^{\text{\,high}} = \sval{res_erdos}{par.s.lam_min_gains(2)}$ (c.f. panel (h) in  Figure~\ref{fig:results_ieeebus}).
% =  \frac{1}{\sval{res_ieebus}{par.s.tr_inv_gains_rec(2)}}$.
%
The entries of the perturbation matrix obtained for the initial value of the penalization parameter~$\gI = \sval{res_ieeebus}{hst.gammas.first}$ were
\begin{align*}
&\;\,\Delta(\gI) = \\
&\left [ \begin{smallmatrix}
\,\,\cdot\,\, & 0.05 & \,\,\cdot\,\, & \,\,\cdot\,\, & 0.22 & \,\,\cdot\,\, & \,\,\cdot\,\, & \,\,\cdot\,\, & \,\,\cdot\,\, & \,\,\cdot\,\, & \,\,\cdot\,\, & \,\,\cdot\,\, & \,\,\cdot\,\, & \,\,\cdot\,\, \\ 
-0.32 & \,\,\cdot\,\, & 0.24 & 0.03 & 0.02 & \,\,\cdot\,\, & \,\,\cdot\,\, & \,\,\cdot\,\, & \,\,\cdot\,\, & \,\,\cdot\,\, & \,\,\cdot\,\, & \,\,\cdot\,\, & \,\,\cdot\,\, & \,\,\cdot\,\, \\ 
\,\,\cdot\,\, & 0.34 & \,\,\cdot\,\, & 0.06 & \,\,\cdot\,\, & \,\,\cdot\,\, & \,\,\cdot\,\, & \,\,\cdot\,\, & \,\,\cdot\,\, & \,\,\cdot\,\, & \,\,\cdot\,\, & \,\,\cdot\,\, & \,\,\cdot\,\, & \,\,\cdot\,\, \\ 
\,\,\cdot\,\, & -0.01 & \,\,*\,\, & \,\,\cdot\,\, & -0.00 & \,\,\cdot\,\, & \,\,*\,\, & \,\,\cdot\,\, & \,\,*\,\, & \,\,\cdot\,\, & \,\,\cdot\,\, & \,\,\cdot\,\, & \,\,\cdot\,\, & \,\,\cdot\,\, \\ 
-0.00 & -0.02 & \,\,\cdot\,\, & -0.03 & \,\,\cdot\,\, & \,\,*\,\, & \,\,\cdot\,\, & \,\,\cdot\,\, & \,\,\cdot\,\, & \,\,\cdot\,\, & \,\,\cdot\,\, & \,\,\cdot\,\, & \,\,\cdot\,\, & \,\,\cdot\,\, \\ 
\,\,\cdot\,\, & \,\,\cdot\,\, & \,\,\cdot\,\, & \,\,\cdot\,\, & -0.03 & \,\,\cdot\,\, & \,\,\cdot\,\, & \,\,\cdot\,\, & \,\,\cdot\,\, & \,\,\cdot\,\, & \,\,*\,\, & \,\,*\,\, & \,\,*\,\, & \,\,\cdot\,\, \\ 
\,\,\cdot\,\, & \,\,\cdot\,\, & \,\,\cdot\,\, & -0.02 & \,\,\cdot\,\, & \,\,\cdot\,\, & \,\,\cdot\,\, & \,\,*\,\, & \,\,*\,\, & \,\,\cdot\,\, & \,\,\cdot\,\, & \,\,\cdot\,\, & \,\,\cdot\,\, & \,\,\cdot\,\, \\ 
\,\,\cdot\,\, & \,\,\cdot\,\, & \,\,\cdot\,\, & \,\,\cdot\,\, & \,\,\cdot\,\, & \,\,\cdot\,\, & \,\,*\,\, & \,\,\cdot\,\, & \,\,\cdot\,\, & \,\,\cdot\,\, & \,\,\cdot\,\, & \,\,\cdot\,\, & \,\,\cdot\,\, & \,\,\cdot\,\, \\ 
\,\,\cdot\,\, & \,\,\cdot\,\, & \,\,\cdot\,\, & -0.05 & \,\,\cdot\,\, & \,\,\cdot\,\, & \,\,*\,\, & \,\,\cdot\,\, & \,\,\cdot\,\, & \,\,*\,\, & \,\,\cdot\,\, & \,\,\cdot\,\, & \,\,\cdot\,\, & 0.01 \\ 
\,\,\cdot\,\, & \,\,\cdot\,\, & \,\,\cdot\,\, & \,\,\cdot\,\, & \,\,\cdot\,\, & \,\,\cdot\,\, & \,\,\cdot\,\, & \,\,\cdot\,\, & \,\,*\,\, & \,\,\cdot\,\, & \,\,*\,\, & \,\,\cdot\,\, & \,\,\cdot\,\, & \,\,\cdot\,\, \\ 
\,\,\cdot\,\, & \,\,\cdot\,\, & \,\,\cdot\,\, & \,\,\cdot\,\, & \,\,\cdot\,\, & \,\,*\,\, & \,\,\cdot\,\, & \,\,\cdot\,\, & \,\,\cdot\,\, & \,\,*\,\, & \,\,\cdot\,\, & \,\,\cdot\,\, & \,\,\cdot\,\, & \,\,\cdot\,\, \\ 
\,\,\cdot\,\, & \,\,\cdot\,\, & \,\,\cdot\,\, & \,\,\cdot\,\, & \,\,\cdot\,\, & \,\,*\,\, & \,\,\cdot\,\, & \,\,\cdot\,\, & \,\,\cdot\,\, & \,\,\cdot\,\, & \,\,\cdot\,\, & \,\,\cdot\,\, & \,\,*\,\, & \,\,\cdot\,\, \\ 
\,\,\cdot\,\, & \,\,\cdot\,\, & \,\,\cdot\,\, & \,\,\cdot\,\, & \,\,\cdot\,\, & \,\,*\,\, & \,\,\cdot\,\, & \,\,\cdot\,\, & \,\,\cdot\,\, & \,\,\cdot\,\, & \,\,\cdot\,\, & \,\,*\,\, & \,\,\cdot\,\, & \,\,*\,\, \\ 
\,\,\cdot\,\, & \,\,\cdot\,\, & \,\,\cdot\,\, & \,\,\cdot\,\, & \,\,\cdot\,\, & \,\,\cdot\,\, & \,\,\cdot\,\, & \,\,\cdot\,\, & \,\,*\,\, & \,\,\cdot\,\, & \,\,\cdot\,\, & \,\,\cdot\,\, & \,\,*\,\, & \,\,\cdot\,\, \\ 
\end{smallmatrix}\right ] 
.
\end{align*}
Here, the symbol `$\ast$' means that the specific entry had a value approximately zero (i.e., within a tolerance $\epsilon_\text{s} =  \sval{res_ieeebus}{par.m.tol_sparsity}$), even though the original network topology and sparsity constraints allowed a non-zero \perturbation{} value. More specifically,
$\sval{res_ieeebus}{hst.card_first}$
out of 
$\sval{res_ieeebus}{hst.card_A0}$
non-zero possible entries were used.
The algorithm was executed for increasing values of $\gamma$ until the stopping criterion was met, in particular, occurring for $\gL = \sval{res_ieeebus}{hst.gammas.stopped}$. The penalized  values obtained in this case were given by
\begin{align*}
&\;\,\Delta(\gL) = \\
&\left [ \begin{smallmatrix}
\,\,\cdot\,\, & \,\,\circledast\,\, & \,\,\cdot\,\, & \,\,\cdot\,\, & 0.23 & \,\,\cdot\,\, & \,\,\cdot\,\, & \,\,\cdot\,\, & \,\,\cdot\,\, & \,\,\cdot\,\, & \,\,\cdot\,\, & \,\,\cdot\,\, & \,\,\cdot\,\, & \,\,\cdot\,\, \\ 
-0.34 & \,\,\cdot\,\, & \,\,\circledast\,\, & 0.09 & 0.02 & \,\,\cdot\,\, & \,\,\cdot\,\, & \,\,\cdot\,\, & \,\,\cdot\,\, & \,\,\cdot\,\, & \,\,\cdot\,\, & \,\,\cdot\,\, & \,\,\cdot\,\, & \,\,\cdot\,\, \\ 
\,\,\cdot\,\, & 0.28 & \,\,\cdot\,\, & \,\,\circledast\,\, & \,\,\cdot\,\, & \,\,\cdot\,\, & \,\,\cdot\,\, & \,\,\cdot\,\, & \,\,\cdot\,\, & \,\,\cdot\,\, & \,\,\cdot\,\, & \,\,\cdot\,\, & \,\,\cdot\,\, & \,\,\cdot\,\, \\ 
\,\,\cdot\,\, & \,\,\circledast\,\, & \,\,*\,\, & \,\,\cdot\,\, & \,\,\circledast\,\, & \,\,\cdot\,\, & \,\,*\,\, & \,\,\cdot\,\, & \,\,*\,\, & \,\,\cdot\,\, & \,\,\cdot\,\, & \,\,\cdot\,\, & \,\,\cdot\,\, & \,\,\cdot\,\, \\ 
\,\,\circledast\,\, & \,\,\circledast\,\, & \,\,\cdot\,\, & \,\,\circledast\,\, & \,\,\cdot\,\, & \,\,*\,\, & \,\,\cdot\,\, & \,\,\cdot\,\, & \,\,\cdot\,\, & \,\,\cdot\,\, & \,\,\cdot\,\, & \,\,\cdot\,\, & \,\,\cdot\,\, & \,\,\cdot\,\, \\ 
\,\,\cdot\,\, & \,\,\cdot\,\, & \,\,\cdot\,\, & \,\,\cdot\,\, & \,\,\circledast\,\, & \,\,\cdot\,\, & \,\,\cdot\,\, & \,\,\cdot\,\, & \,\,\cdot\,\, & \,\,\cdot\,\, & \,\,*\,\, & \,\,*\,\, & \,\,*\,\, & \,\,\cdot\,\, \\ 
\,\,\cdot\,\, & \,\,\cdot\,\, & \,\,\cdot\,\, & \,\,\circledast\,\, & \,\,\cdot\,\, & \,\,\cdot\,\, & \,\,\cdot\,\, & \,\,*\,\, & \,\,*\,\, & \,\,\cdot\,\, & \,\,\cdot\,\, & \,\,\cdot\,\, & \,\,\cdot\,\, & \,\,\cdot\,\, \\ 
\,\,\cdot\,\, & \,\,\cdot\,\, & \,\,\cdot\,\, & \,\,\cdot\,\, & \,\,\cdot\,\, & \,\,\cdot\,\, & \,\,*\,\, & \,\,\cdot\,\, & \,\,\cdot\,\, & \,\,\cdot\,\, & \,\,\cdot\,\, & \,\,\cdot\,\, & \,\,\cdot\,\, & \,\,\cdot\,\, \\ 
\,\,\cdot\,\, & \,\,\cdot\,\, & \,\,\cdot\,\, & \,\,\circledast\,\, & \,\,\cdot\,\, & \,\,\cdot\,\, & \,\,*\,\, & \,\,\cdot\,\, & \,\,\cdot\,\, & \,\,*\,\, & \,\,\cdot\,\, & \,\,\cdot\,\, & \,\,\cdot\,\, & \,\,\circledast\,\, \\ 
\,\,\cdot\,\, & \,\,\cdot\,\, & \,\,\cdot\,\, & \,\,\cdot\,\, & \,\,\cdot\,\, & \,\,\cdot\,\, & \,\,\cdot\,\, & \,\,\cdot\,\, & \,\,*\,\, & \,\,\cdot\,\, & \,\,*\,\, & \,\,\cdot\,\, & \,\,\cdot\,\, & \,\,\cdot\,\, \\ 
\,\,\cdot\,\, & \,\,\cdot\,\, & \,\,\cdot\,\, & \,\,\cdot\,\, & \,\,\cdot\,\, & \,\,*\,\, & \,\,\cdot\,\, & \,\,\cdot\,\, & \,\,\cdot\,\, & \,\,*\,\, & \,\,\cdot\,\, & \,\,\cdot\,\, & \,\,\cdot\,\, & \,\,\cdot\,\, \\ 
\,\,\cdot\,\, & \,\,\cdot\,\, & \,\,\cdot\,\, & \,\,\cdot\,\, & \,\,\cdot\,\, & \,\,*\,\, & \,\,\cdot\,\, & \,\,\cdot\,\, & \,\,\cdot\,\, & \,\,\cdot\,\, & \,\,\cdot\,\, & \,\,\cdot\,\, & \,\,*\,\, & \,\,\cdot\,\, \\ 
\,\,\cdot\,\, & \,\,\cdot\,\, & \,\,\cdot\,\, & \,\,\cdot\,\, & \,\,\cdot\,\, & \,\,*\,\, & \,\,\cdot\,\, & \,\,\cdot\,\, & \,\,\cdot\,\, & \,\,\cdot\,\, & \,\,\cdot\,\, & \,\,*\,\, & \,\,\cdot\,\, & \,\,*\,\, \\ 
\,\,\cdot\,\, & \,\,\cdot\,\, & \,\,\cdot\,\, & \,\,\cdot\,\, & \,\,\cdot\,\, & \,\,\cdot\,\, & \,\,\cdot\,\, & \,\,\cdot\,\, & \,\,*\,\, & \,\,\cdot\,\, & \,\,\cdot\,\, & \,\,\cdot\,\, & \,\,*\,\, & \,\,\cdot\,\, \\ 
\end{smallmatrix}\right ] 
.
\end{align*}
Here, the symbol `$\circledast$' indicates that the corresponding entry resulted in an approximately zero value (i.e., within a tolerance 
$\epsilon_\text{s} =  \sval{res_ieeebus}{par.m.tol_sparsity}$) for this value of $\gL$, whereas the same entry took a nonzero value when the penalization weight $\gI$ was considered. In particular, 
while $\sval{res_ieeebus}{hst.card_first}$
nonzero entries were used for 
$\gI$, 
this number was reduced to
$\sval{res_ieeebus}{hst.card_last}$ for
$\gL$, as a result of effect of the structural penalty.

%===============================================
\section{Conclusion}

In this paper, we have formulated and solved two problems involving the tuning of edge weights in a given discrete-time networked dynamical system such that certain reachability requirements, defined in terms of the reachability Gramian, are satisfied. In our first problem, we aimed at finding a feasible tuning of the edge weights. A direct formulation of this problems results in highly nonlinear optimization program. In order to overcome this challenge, we proposed a chain of transformations allowing us to reformulate this problem as an optimization program involving a rank constraint over a structured matrix presenting an affine dependence on the decision variables. We then relax this rank constraint using a truncated nuclear norm and proposed a sequence of convex programs to solve this relaxation. Furthermore, we have also considered a second problem in which we aimed at finding edge-weights in order to satisfy certain reachability requirements while tuning a small number of edges. Our computational approach to solve these problems has been illustrated with several numerical experiments. As future work, we plan to examine a more comprehensive class of systems, including bilinear and stochastic systems, through their corresponding reachability Gramians. \resub{Another interesting avenue of investigation would be to provide insights on the graph-theoretic characteristics of optimal designs produced for different network topologies. }

%%%%%%%%%%%%%%%%%%%%%%%%%%%%%%%%%%%%%%%%%%%%%%%%%%%

%%===============================================
%\section*{Acknowledgment}

%%%%%%%%%%%%%%%%%%%%%%%%%%%%%%%%%%%%%%%%%%%%%%%%%%%
\appendices


\bibliographystyle{IEEEtran}
\bibliography{controllability_design}

\begin{thebibliography}{10}
\providecommand{\url}[1]{#1}
\csname url@samestyle\endcsname
\providecommand{\newblock}{\relax}
\providecommand{\bibinfo}[2]{#2}
\providecommand{\BIBentrySTDinterwordspacing}{\spaceskip=0pt\relax}
\providecommand{\BIBentryALTinterwordstretchfactor}{4}
\providecommand{\BIBentryALTinterwordspacing}{\spaceskip=\fontdimen2\font plus
\BIBentryALTinterwordstretchfactor\fontdimen3\font minus
  \fontdimen4\font\relax}
\providecommand{\BIBforeignlanguage}[2]{{%
\expandafter\ifx\csname l@#1\endcsname\relax
\typeout{** WARNING: IEEEtran.bst: No hyphenation pattern has been}%
\typeout{** loaded for the language `#1'. Using the pattern for}%
\typeout{** the default language instead.}%
\else
\language=\csname l@#1\endcsname
\fi
#2}}
\providecommand{\BIBdecl}{\relax}
\BIBdecl

\bibitem{FB-LNS}
\BIBentryALTinterwordspacing
F.~Bullo, \emph{Lectures on Network Systems}, 1st~ed.\hskip 1em plus 0.5em
  minus 0.4em\relax CreateSpace, 2018, with contributions by J. Cortes, F.
  Dorfler, and S. Martinez. [Online]. Available:
  \url{http://motion.me.ucsb.edu/book-lns}
\BIBentrySTDinterwordspacing

\bibitem{hespanha2009linear}
J.~Hespanha, \emph{Linear {S}ystems {T}heory}.\hskip 1em plus 0.5em minus
  0.4em\relax Princeton {U}niversity {P}ress, 2009.

\bibitem{clark2012}
A.~Clark, L.~Bushnell, and R.~Poovendran, ``On {L}eader {S}election for
  {P}erformance and {C}ontrollability in {M}ulti-agent {S}ystems,'' in
  \emph{Proceedings of the 51st Annual Conference on Decision and
  Control}.\hskip 1em plus 0.5em minus 0.4em\relax IEEE, Dec 2012, pp. 86--93.

\bibitem{6580798}
A.~Chapman and M.~Mesbahi, ``On {S}trong {S}tructural {C}ontrollability of
  {N}etworked {S}ystems: A {C}onstrained {M}atching {A}pproach,'' in
  \emph{Proceedings of the 52nd American Control Conference}.\hskip 1em plus
  0.5em minus 0.4em\relax IEEE, 2013, pp. 6126--6131.

\bibitem{summers2016actuator}
T.~Summers, ``Actuator {P}lacement in {N}etworks {U}sing {O}ptimal {C}ontrol
  {P}erformance {M}etrics,'' in \emph{Proceedings of the 55th Annual Conference
  on Decision and Control}.\hskip 1em plus 0.5em minus 0.4em\relax IEEE, 2016,
  pp. 2703--2708.

\bibitem{summers2014submodularity}
T.~H. Summers, F.~L. Cortesi, and J.~Lygeros, ``On {S}ubmodularity and
  {C}ontrollability in {C}omplex {D}ynamical {N}etworks,'' \emph{IEEE
  Transactions on Control of Network Systems}, vol.~3, no.~1, pp. 91--101,
  2016.

\bibitem{pequito2017robust}
S.~Pequito, G.~Ramos, S.~Kar, A.~P. Aguiar, and J.~Ramos, ``The {R}obust
  {M}inimal {C}ontrollability {P}roblem,'' \emph{Automatica}, vol.~82, pp.
  261--268, 2017.

\bibitem{olshevsky2014minimal}
A.~Olshevsky, ``Minimal {C}ontrollability {P}roblems,'' \emph{IEEE Transactions
  on Control of Network Systems}, vol.~1, no.~3, pp. 249--258, 2014.

\bibitem{PequitoJournal}
S.~Pequito, S.~Kar, and A.~P. Aguiar, ``A {F}ramework for {S}tructural
  {I}nput/{O}utput and {C}ontrol {C}onfiguration {S}election in {L}arge-scale
  {S}ystems,'' \emph{IEEE Transactions on Automatic Control}, vol.~61, no.~2,
  pp. 303--318, 2016.

\bibitem{tzoumas2016minimal}
V.~Tzoumas, M.~A. Rahimian, G.~J. Pappas, and A.~Jadbabaie, ``Minimal
  {A}ctuator {P}lacement with {B}ounds on {C}ontrol {E}ffort,'' \emph{IEEE
  Transactions on Control of Network Systems}, vol.~3, no.~1, pp. 67--78, 2016.

\bibitem{pequito2016minimum}
S.~Pequito, S.~Kar, and A.~P. Aguiar, ``Minimum {C}ost {I}nput/{O}utput
  {D}esign for {L}arge-scale {L}inear {S}tructural {S}ystems,''
  \emph{Automatica}, vol.~68, pp. 384--391, 2016.

\bibitem{enyioha2014controllability}
C.~Enyioha, M.~A. Rahimian, G.~J. Pappas, and A.~Jadbabaie, ``Controllability
  and {F}raction of {L}eaders in {I}nfinite {N}etworks,'' in \emph{Proceedings
  of the 53rd Annual Conference on Decision and Control}.\hskip 1em plus 0.5em
  minus 0.4em\relax IEEE, 2014, pp. 1359--1364.

\bibitem{ilic2000dynamics}
M.~D. Ilic and J.~Zaborszky, \emph{Dynamics and {C}ontrol of {L}arge {E}lectric
  {P}ower {S}ystems}.\hskip 1em plus 0.5em minus 0.4em\relax Wiley New York,
  2000.

\bibitem{zhang2012flexible}
X.-P. Zhang, C.~Rehtanz, and B.~Pal, \emph{Flexible AC {T}ransmission
  {S}ystems: {M}odelling and {C}ontrol}.\hskip 1em plus 0.5em minus 0.4em\relax
  Springer Science \& Business Media, 2012.

\bibitem{xiao2004fast}
L.~Xiao and S.~Boyd, ``Fast {L}inear {I}terations for {D}istributed
  {A}veraging,'' \emph{Systems \& Control Letters}, vol.~53, no.~1, pp. 65--78,
  2004.

\bibitem{pasqualetti2014controllability}
F.~Pasqualetti, S.~Zampieri, and F.~Bullo, ``Controllability {M}etrics,
  {L}imitations and {A}lgorithms for {C}omplex {N}etworks,'' \emph{IEEE
  Transactions on Control of Network Systems}, vol.~1, no.~1, pp. 40--52, 2014.

\bibitem{bianchin2015role}
G.~Bianchin, F.~Pasqualetti, and S.~Zampieri, ``The {R}ole of {D}iameter in the
  {C}ontrollability of {C}omplex {N}etworks,'' in \emph{Proceedings of the 54th
  Annual Conference on Decision and Control}.\hskip 1em plus 0.5em minus
  0.4em\relax IEEE, 2015, pp. 980--985.

\bibitem{aguilar2016almost}
C.~O. Aguilar and B.~Gharesifard, ``On {A}lmost {E}quitable {P}artitions and
  {N}etwork {C}ontrollability,'' in \emph{Proceedings of the 55th American
  Control Conference}.\hskip 1em plus 0.5em minus 0.4em\relax IEEE, 2016, pp.
  179--184.

\bibitem{aguilar2015graph}
------, ``Graph {C}ontrollability {C}lasses for the {L}aplacian
  {L}eader-follower {D}ynamics,'' \emph{IEEE Transactions on Automatic
  Control}, vol.~60, no.~6, pp. 1611--1623, 2015.

\bibitem{parlangeli2012reachability}
G.~Parlangeli and G.~Notarstefano, ``On the {R}eachability and {O}bservability
  of {P}ath and {C}ycle {G}raphs,'' \emph{IEEE Transactions on Automatic
  Control}, vol.~57, no.~3, pp. 743--748, 2012.

\bibitem{notarstefano2013controllability}
G.~Notarstefano and G.~Parlangeli, ``Controllability and {O}bservability of
  {G}rid {G}raphs via {R}eduction and {S}ymmetries,'' \emph{IEEE Transactions
  on Automatic Control}, vol.~58, no.~7, pp. 1719--1731, 2013.

\bibitem{chapman2014controllability}
A.~Chapman, M.~Nabi-Abdolyousefi, and M.~Mesbahi, ``Controllability and
  {O}bservability of {N}etwork-of-networks via {C}artesian {P}roducts,''
  \emph{IEEE Transactions on Automatic Control}, vol.~59, no.~10, pp.
  2668--2679, 2014.

\bibitem{tanner2004controllability}
H.~G. Tanner, ``On the {C}ontrollability of {N}earest {N}eighbor
  {I}nterconnections,'' in \emph{Proceedings of the 43rd Annual Conference on
  Decision and Control}, vol.~3.\hskip 1em plus 0.5em minus 0.4em\relax IEEE,
  2004, pp. 2467--2472.

\bibitem{roy2019controllability}
S.~Roy and M.~Xue, ``Controllability-gramian {S}ubmatrices for a {N}etwork
  {C}onsensus {M}odel,'' \emph{arXiv preprint arXiv:1903.09125}, 2019.

\bibitem{zhao2017discrete}
S.~Zhao and F.~Pasqualetti, ``Discrete-time {D}ynamical {N}etworks with
  {D}iagonal {C}ontrollability {G}ramian,'' \emph{IFAC-PapersOnLine}, vol.~50,
  no.~1, pp. 8297--8302, 2017.

\bibitem{zhao2017gramian}
Y.~Zhao and J.~Cort{\'e}s, ``Gramian-based {R}eachability {M}etrics for
  {B}ilinear {N}etworks,'' \emph{IEEE Transactions on Control of Network
  Systems}, vol.~4, no.~3, pp. 620--631, 2017.

\bibitem{bianchin2016observability}
G.~Bianchin, P.~Frasca, A.~Gasparri, and F.~Pasqualetti, ``The {O}bservability
  {R}adius of {N}etworks,'' \emph{IEEE Transactions on Automatic Control},
  2016.

\bibitem{siami2018growing}
M.~Siami and N.~Motee, ``Growing {L}inear {D}ynamical {N}etworks {E}ndowed by
  {S}pectral {S}ystemic {P}erformance {M}easures,'' \emph{IEEE Transactions on
  Automatic Control}, vol.~63, no.~7, 2018.

\bibitem{shafi2011graph}
S.~Y. Shafi, M.~Arcak, and L.~El~Ghaoui, ``Graph {W}eight {A}llocation to
  {M}eet {L}aplacian {S}pectral {C}onstraints,'' \emph{IEEE Transactions on
  Automatic Control}, vol.~57, no.~7, pp. 1872--1877, 2011.

\bibitem{torres2018dominant}
J.~A. Torres and S.~Roy, ``Dominant {E}igenvalue {M}inimization with {T}race
  {P}reserving {D}iagonal {P}erturbation: {S}ubset {D}esign {P}roblem,''
  \emph{Automatica}, vol.~89, pp. 160--168, 2018.

\bibitem{preciado2016distributed}
V.~M. Preciado and M.~M. Zavlanos, ``Distributed {N}etwork {D}esign for
  {L}aplacian {E}igenvalue {P}lacement,'' \emph{IEEE Transactions on Control of
  Network Systems}, vol.~4, no.~3, pp. 598--609, 2016.

\bibitem{sun2018weighted}
C.~Sun, R.~Dai, and M.~Mesbahi, ``Weighted {N}etwork {D}esign with
  {C}ardinality {C}onstraints via {A}lternating {D}irection {M}ethod of
  {M}ultipliers,'' \emph{IEEE Transactions on Control of Network Systems},
  vol.~5, no.~4, pp. 2073--2084, 2018.

\bibitem{hassan2017topology}
S.~Hassan-Moghaddam and M.~R. Jovanovi{\'c}, ``Topology {D}esign for
  {S}tochastically {F}orced {C}onsensus {N}etworks,'' \emph{IEEE Transactions
  on Control of Network Systems}, vol.~5, no.~3, pp. 1075--1086, 2017.

\bibitem{preciado2013optimal}
V.~M. Preciado, M.~Zargham, C.~Enyioha, A.~Jadbabaie, and G.~Pappas, ``Optimal
  {V}accine {A}llocation to {C}ontrol {E}pidemic {O}utbreaks in {A}rbitrary
  {N}etworks,'' in \emph{Proceedings of the 52nd Annual Conference on Decision
  and Control}.\hskip 1em plus 0.5em minus 0.4em\relax IEEE, 2013, pp.
  7486--7491.

\bibitem{preciado2014optimal}
V.~M. Preciado, M.~Zargham, C.~Enyioha, A.~Jadbabaie, and G.~J. Pappas,
  ``Optimal {R}esource {A}llocation for {N}etwork {P}rotection {A}gainst
  {S}preading {P}rocesses,'' \emph{IEEE Transactions on Control of Network
  Systems}, vol.~1, no.~1, pp. 99--108, 2014.

\bibitem{pajic2011wireless}
M.~Pajic, S.~Sundaram, G.~J. Pappas, and R.~Mangharam, ``The {W}ireless
  {C}ontrol {N}etwork: {A} {N}ew {A}pproach for {C}ontrol over {N}etworks,''
  \emph{IEEE Transactions on Automatic Control}, vol.~56, no.~10, pp.
  2305--2318, 2011.

\bibitem{wan2008designing}
Y.~Wan, S.~Roy, and A.~Saberi, ``Designing {S}patially {H}eterogeneous
  {S}trategies for {C}control of {V}irus {S}pread,'' \emph{IET Systems
  Biology}, vol.~2, no.~4, pp. 184--201, 2008.

\bibitem{becker2017network}
C.~O. Becker, S.~Pequito, G.~J. Pappas, and V.~M. Preciado, ``Network {D}esign
  for {C}ontrollability {M}etrics,'' in \emph{Proceedings of the 56th Annual
  Conference on Decision and Control}.\hskip 1em plus 0.5em minus 0.4em\relax
  IEEE, 2017, pp. 4193--4198.

\bibitem{blekherman2012semidefinite}
G.~Blekherman, P.~A. Parrilo, and R.~R. Thomas, \emph{Semidefinite
  {O}ptimization and {C}onvex {A}lgebraic {G}eometry}.\hskip 1em plus 0.5em
  minus 0.4em\relax SIAM, 2012.

\bibitem{muller1972analysis}
P.~M{\"u}ller and H.~Weber, ``Analysis and {O}ptimization of {C}ertain
  {Q}ualities of {C}ontrollability and {O}bservability for {L}inear {D}ynamical
  {S}ystems,'' \emph{Automatica}, vol.~8, no.~3, pp. 237--246, 1972.

\bibitem{donoho2006compressed}
D.~L. Donoho, ``Compressed {S}ensing,'' \emph{IEEE Transactions on Information
  Theory}, vol.~52, no.~4, pp. 1289--1306, 2006.

\bibitem{recht2010guaranteed}
B.~Recht, M.~Fazel, and P.~A. Parrilo, ``Guaranteed {M}inimum-rank {S}olutions
  of {L}inear {M}atrix {E}quations via {N}uclear {N}orm {M}inimization,''
  \emph{SIAM review}, vol.~52, no.~3, pp. 471--501, 2010.

\bibitem{hastie2015statistical}
T.~Hastie, R.~Tibshirani, and M.~Wainwright, \emph{Statistical {L}earning with
  {S}parsity: the {L}asso and {G}eneralizations}.\hskip 1em plus 0.5em minus
  0.4em\relax Chapman and Hall/CRC, 2015.

\bibitem{dorfler2014sparsity}
F.~D{\"o}rfler, M.~R. Jovanovi{\'c}, M.~Chertkov, and F.~Bullo,
  ``Sparsity-{P}romoting {O}ptimal {W}ide-{A}rea {C}ontrol of {P}ower
  {N}etworks,'' \emph{IEEE Transactions on Power Systems}, vol.~29, no.~5, pp.
  2281--2291, 2014.

\bibitem{lin2013design}
F.~Lin, M.~Fardad, and M.~R. Jovanovi{\'c}, ``Design of {O}ptimal {S}parse
  {F}eedback {G}ains via the {A}lternating {D}irection {M}ethod of
  {M}ultipliers,'' \emph{IEEE Transactions on Automatic Control}, vol.~58,
  no.~9, pp. 2426--2431, 2013.

\bibitem{blomberg2014approximate}
N.~Blomberg, C.~R. Rojas, and B.~Wahlberg, ``Approximate {R}egularization
  {P}ath for {N}uclear {N}orm {B}ased {H}2 {M}odel {R}eduction,'' in
  \emph{Proceedings of the 53rd Annual Conference on Decision and
  Control}.\hskip 1em plus 0.5em minus 0.4em\relax IEEE, 2014, pp. 3637--3641.

\bibitem{davison1973properties}
E.~Davison and S.~Wang, ``Properties of {L}inear {T}ime-invariant
  {M}ultivariable {S}ystems {S}ubject to {A}rbitrary {O}utput and {S}tate
  {F}eedback,'' \emph{IEEE Transactions on Automatic Control}, vol.~18, no.~1,
  pp. 24--32, 1973.

\bibitem{shields1976structural}
R.~Shields and J.~Pearson, ``Structural {C}ontrollability of {M}ultiinput
  {L}inear {S}ystems,'' \emph{IEEE Transactions on Automatic control}, vol.~21,
  no.~2, pp. 203--212, 1976.

\bibitem{dion2003generic}
J.-M. Dion, C.~Commault, and J.~Van Der~Woude, ``Generic {P}roperties and
  {C}ontrol of {L}inear {S}tructured {S}ystems: a {S}urvey,''
  \emph{Automatica}, vol.~39, no.~7, pp. 1125--1144, 2003.

\bibitem{menara2018structural}
T.~Menara, D.~Bassett, and F.~Pasqualetti, ``Structural {C}ontrollability of
  {S}ymmetric {N}etworks,'' \emph{IEEE Transactions on Automatic Control},
  2018.

\bibitem{zhang2006schur}
F.~Zhang, \emph{The {S}chur {Co}mplement and {I}ts {A}pplications}.\hskip 1em
  plus 0.5em minus 0.4em\relax Springer Science \& Business Media, 2006,
  vol.~4.

\bibitem{hu2012fast}
Y.~Hu, D.~Zhang, J.~Ye, X.~Li, and X.~He, ``Fast and {A}ccurate {M}atrix
  {C}ompletion via {T}runcated {N}uclear {N}orm {R}egularization,'' \emph{IEEE
  Transactions on Pattern Analysis and Machine Intelligence}, p.~1, 2012.

\bibitem{giesen2012regularization}
J.~Giesen, M.~Jaggi, and S.~Laue, ``Regularization {P}aths with {G}uarantees
  for {C}onvex {S}emidefinite {O}ptimization,'' in \emph{Artificial
  Intelligence and Statistics}, 2012, pp. 432--439.

\bibitem{christie2000power}
R.~Christie, ``Power {S}ystems {T}est {C}ase {A}rchive,'' \emph{Electrical
  Engineering Dept., University of Washington}, 2000.

\bibitem{fuchs2013actuator}
A.~Fuchs and M.~Morari, ``Actuator {P}erformance {E}valuation using {LMI}s for
  {O}ptimal {HVDC} {P}lacement,'' in \emph{Proceedings of the European Control
  Conference}.\hskip 1em plus 0.5em minus 0.4em\relax IEEE, 2013, pp.
  1529--1534.

\bibitem{github2018netdeco}
C.~O. Becker, S.~Pequito, G.~J. Pappas, and V.~M. Preciado, ``Online repository
  with code for paper {N}etwork {D}esign for {C}ontrollability {M}etrics,''
  \url{https://github.com/cassianobecker/netdeco}, 2019.

\bibitem{sousa2018}
\BIBentryALTinterwordspacing
T.~Sousa, T.~Soares, P.~Pinson, F.~Moret, T.~Baroche, and E.~Sorin, ``The
  {P2P}-{IEEE} 14 {B}us {S}ystem {D}ata {S}et,'' Apr. 2018. [Online].
  Available: \url{https://doi.org/10.5281/zenodo.1220935}
\BIBentrySTDinterwordspacing

\bibitem{wonham1985geometric}
W.~M. Wonham, \emph{Linear {M}ultivariable {C}ontrol: {A} {G}eometric
  {A}pproach}.\hskip 1em plus 0.5em minus 0.4em\relax Springer-Verlag New
  York., 1985.

\bibitem{horn1990matrix}
R.~A. Horn, R.~A. Horn, and C.~R. Johnson, \emph{Matrix {A}nalysis}.\hskip 1em
  plus 0.5em minus 0.4em\relax Cambridge {U}niversity {P}ress, 1990.

\bibitem{razaviyayn2013unified}
M.~Razaviyayn, M.~Hong, and Z.-Q. Luo, ``A {U}nified {C}onvergence {A}nalysis
  of {B}lock {S}uccessive {M}inimization {M}ethods for {N}onsmooth
  {O}ptimization,'' \emph{SIAM Journal on Optimization}, vol.~23, no.~2, pp.
  1126--1153, 2013.

\bibitem{calafiore2014optimization}
G.~C. Calafiore and L.~El~Ghaoui, \emph{Optimization {M}odels}.\hskip 1em plus
  0.5em minus 0.4em\relax Cambridge {U}niversity {P}ress, 2014.

\bibitem{watson1992characterization}
G.~A. Watson, ``Characterization of the {S}ubdifferential of {S}ome {M}atrix
  {N}orms,'' \emph{Linear {A}lgebra and its {A}pplications}, vol. 170, pp.
  33--45, 1992.

\end{thebibliography}

\normalsize
%\newpage
%----------------------------------------------------------------------------------
\section{Additional Lemmas}
\label{sec:appaffine}

\newcommand{\Akron}{ (A \otimes A  - I_{n^2} )}

\resub{
\begin{lemmaa}(Uniqueness for the Lyapunov equation)
	\label{thm:generality_uniqueness} A solution~$W \in \Ss^n$ to 
	\begin{align}
		AWA^{\tp}-W & =-BB^{\tp} \label{eq:dlyap}
	\end{align}	
	exists and is unique for any matrices~$A \equiv A(\G) \in \R^{n\times n}$ and $B \in \R^{n\times m}$ except for a proper algebraic variety~$\varW \subset \R^{|\E|}$, where $|\E|$ is the number of free entries in $A$. 
\end{lemmaa}
\begin{proof}
	Existence and uniqueness of a solution~$W \in \Ss^{n}$ to \eqref{eq:dlyap} can determined by examining the result of applying the vectorization
	operator on both sides to get
	\begin{align}
	%\vc(AWA^{\tp}-W) & =\vc(-BB^{\tp})\nonumber\\
	%	(A\otimes A)\vc(W)-\vc(W) & =-\vc(BB^{\tp})\nonumber\\
	\Akron\vc(W) & =-\vc(BB^{\tp}),\label{eq:vec_dlyap}
	\end{align}
	where the symbol~$\otimes$ denotes the Kronecker product, and the function~$\vc(\cdot)$ is the vectorization operator. Equation~\eqref{eq:vec_dlyap} will have a unique solution whenever the coefficient matrix~$\Akron$ is nonsingular. 
	Following \cite{menara2018structural}, we let $a_\E \coloneqq \left ( [A]_{i,j} : (j,i) \in \E \right )$ represent an ordered set containing the entries of $A$ in lexicographic order. Next, we define a correspondence between $a_\E$ and a vector~$z \in \R^d,\,d=|\E|$, and notice that 
	$\varphi(z) \coloneqq \det \Akron $ is a polynomial function of the components of $z$. Then, we observe that the set~$\varW \coloneqq \{ z \in \R^d : \varphi(z)  = 0\} $
	defines a proper algebraic variety of $\R^d$~\cite{wonham1985geometric} where the matrix $\Akron$ is singular. 
	Therefore, for any matrix~$A$ having entries from the correspondence between $a_\E$ and $z$ such that $z \in \R^d \setminus \varW$, the matrix~$\Akron$ will be nonsingular, and \eqref{eq:dlyap} will have a unique solution~$\vc(W) = -\Akron^{-1}\cdot \vc(BB^{\tp})$.
\end{proof}	
}

\resub{	
\begin{lemmaa}(Stability from the Lyapunov equation)
	\label{thm:lyap_stability_schur}
	Consider the discrete-time Lyapunov equation~\eqref{eq:dlyap} with a unique solution~$W$. If $W\succ 0$ and the pair~$(A,B)$ is reachable, then the matrix~$A$ is Schur stable.
\end{lemmaa}
\begin{proof}
	The proof is a trivial extension to discrete-time systems of the proof to Theorem~12.5 in~\cite[p.103]{hespanha2009linear}. 
	To begin, we pick a left eigenvector~$v$
	of $A$ such that
	% $v^{\tp}A=\lambda v^{\tp}$, i.e., 
	$A^{\tp}v=\lambda v$.
	Then, we compare the quadratic forms for $v$ at both sides
	of (\ref{eq:dlyap}):
	\begin{align}
	v^{\ast}(AWA^{\tp}-W)v & =-v^{\ast}(BB^{\tp})v\nonumber\\
	%	v^{\ast}(AWA^{\tp})v-v^{\ast}Wv & =-(B^{\tp}v)^{\ast}B^{\tp}v\nonumber\\
	%(A^{\tp}v)^{\ast}W(A^{\tp}v)-v^{\ast}Wv & =-(B^{\tp}v)^{\ast}B^{\tp}v\nonumber\\
	%(\lambda v)^{\ast}W(\lambda v)-v^{\ast}Wv & =-(B^{\tp}v)^{\ast}B^{\tp}v\nonumber\\%-\|B^{\tp}v\|^{2}\nonumber\\
	%	(\lambda^{\ast}\lambda)v{}^{\ast}Wv-v^{\ast}Wv & =-\|B^{\tp}v\|^{2}\nonumber\\
	%	(\lambda^{\ast}\lambda-1)v{}^{\ast}Wv & =-\|B^{\tp}v\|^{2}\nonumber\\
	(|\lambda|^{2}-1)v{}^{\ast}Wv & =-\|B^{\tp}v\|^{2},\label{eq:lyap_quadform_last}
	\end{align}
	where $v^\ast$ denotes the conjugate-transpose of $v$.
	Because we assumed that $W\succ0$, it is the case that $v^{\ast}Wv>0$. 
	%Next, we split our analysis in two cases, depending on the controllability of $(A,B)$:
	Then, since $(A,B)$ is reachable by assumption, from the Popov-Belevitch-Hautus (PBH) test for controllability \cite[c.f. Theorem 12.3, p.101]{hespanha2009linear},
	there is no eigenvector~$v$ of $A^{\tp}$ such that $B^{\tp}v=0$.
	Therefore, we have that 
	$\|B^{\tp}v\|^{2}>0$, which implies $|\lambda|<1$ in \eqref{eq:lyap_quadform_last}. Hence,
	the matrix~$A$ is Schur stable.
\end{proof}
}

\begin{lemmaa}[Trace-inverse as semidefinite constraint] 
	\label{lem:tr_inv}
	The condition $n\tau  - \tr\{\left [W\right ]^{-1}\}\geq 0 $
	for $W \in \Spp^n$ can be formulated as a semidefinite constraint  
	requiring the existence of a variable $P\in\R^{n \times n}$ such that
	\begin{align*}
	n\tau-\tr\{P\}  \geq0 \text{ and }
	\begin{bmatrix} W & I_{n}\\
	I_{n} & P
	\end{bmatrix} & \succeq0.
	\end{align*}
\end{lemmaa}
\begin{proof}
	Note that $P - W^{-1} \succeq 0  \Rightarrow \tr\{P\} -  \tr \{W^{-1}\}\geq 0$. Then, applying the Schur complement on $P - W^{-1} \succeq 0$ yields the relationship in terms of the inverse of $W$.
\end{proof}

\begin{lemmaa}[Von Neumann's Trace Inequality]
	\label{lem:von_neumann}	
	For any $X \in \R^{m \times n}$ and pair $(L,R) \in \{L \in \R^{r \times m}, R \in \R^{r \times n}: LL^\tp = I_r, RR^\tp=I_r\}$, where $1\leq r \leq \min\{m,n\}$, we have
	\begin{align}
	\tr\{LXR^\tp\}\leq \sum_{i=1}^{r} \sigma_i(X).
	\label{eq:von_neumann}
	\end{align}
	Further, consider the singular value decomposition $X = U\Sigma V^\tp$, where $U = [u_1,\ldots,u_m]$ and $V=[v_1,\ldots,v_n]$. Then, \eqref{eq:von_neumann} holds with equality if
	$L=[u_1,\ldots,u_r]^\tp$ and $R=[v_1,\ldots,v_r]^\tp$.
\end{lemmaa}
\begin{proof}
	See Theorem 3.1~\cite{hu2012fast} and Theorem 7.4.1.1~\cite[p. 458]{horn1990matrix}.
\end{proof}

\iftoggle{arxiv}{}{
%\newpage
\iftrue
\begin{IEEEbiography}
	[{\includegraphics[width=1in,height=1.25in,clip,keepaspectratio]{figs/profile/cbecker_bw.jpg}}]
	{Cassiano O. Becker}
	is a Ph.D. student in the Department of Electrical and Systems Engineering at the University of Pennsylvania. He received an M.Sc. in Telecommunications (with distinction) from the University College London, an M.S. in Electrical Engineering from the State University of Campinas and a B.Eng. in Electrical Engineering from the Federal University of Rio Grande do Sul, Brazil. Before starting his doctoral studies, he worked as systems engineer and software developer in technology companies such as Harris Corporation and Siemens. 
	His research interests include inference and control over dynamical systems and complex networks using advanced optimization tools, with applications to cyber-physical systems and neuroscience, among others.
\end{IEEEbiography}

\begin{IEEEbiography}
	[{\includegraphics[width=1in,height=1.25in,clip,keepaspectratio]{figs/profile/spequito2_bw.png}}]
	{S\'{e}rgio Pequito}
	(S’09, M’14) is an assistant professor at the Department of Industrial and Systems Engineering at the Rensselaer Polytechnic Institute. From 2014 to 2017, he was a postdoctoral researcher in General Robotics, Automation, Sensing \& Perception Laboratory (GRASP lab) at University of Pennsylvania. He obtained his Ph.D. in Electrical and Computer Engineering from Carnegie Mellon University and Instituto Superior T\`{e}cnico, through the CMU-Portugal program, in 2014. 
%	Previously, he received his B.Sc. and M.Sc. in Applied Mathematics from the Instituto Superior T\'{e}cnico in 2007 and 2009, respectively. 
	Pequito’s research consists of understanding the global qualitative behavior of large-scale systems from their structural or parametric descriptions and provide a rigorous framework for the design, analysis, optimization and control of large scale (real-world) systems. 
%\	Currently, his interests span to neuroscience and biomedicine, where dynamical systems and control theoretic tools can be leveraged to develop new analysis tools for brain dynamics that, ultimately, will lead to new diagnostics and treatments of neural disorders. 
%	In addition, these tools can be used towards effective personalized medicine and improve brain-computer and brain-machine-brain interfaces that will improve people’s life quality.
	Pequito was awarded the best student paper finalist in the 48th IEEE Conference on Decision and Control (2009). Also, Pequito received the ECE Outstanding Teaching Assistant Award from the Electrical and Computer Engineering Department at Carnegie Mellon University, and the Carnegie Mellon Graduate Teaching Award (University-wide) honorable mention, both in 2012. Also, Pequito was a 2016 EECI European Ph.D. Award on Control for Complex and Heterogeneous Systems finalist and received the 2016 O. Hugo Schuck Award in the Theory Category.
\end{IEEEbiography}
%\vfill
%\newpage
\begin{IEEEbiography}
	[{\includegraphics[width=1in,height=1.25in,clip,keepaspectratio]{figs/profile/gpappas_bw.png}}]
	{George J. Pappas} (S’90,M’91,SM’04,F’09) received the Ph.D. degree in electrical engineering and computer sciences from the University of California, Berkeley, CA, USA, in 1998. He is currently the Joseph Moore Professor and Chair of the Department of Electrical and Systems Engineering, University of Pennsylvania, Philadelphia, PA, USA. He also holds a secondary appointment with the Department of Computer and Information Sciences and the Department of Mechanical Engineering and Applied Mechanics. He is a Member of the GRASP
	Lab and the PRECISE Center. He had previously served as the Deputy Dean for Research with the School of Engineering and Applied Science. His research interests include control theory and, in particular, hybrid systems, embedded systems, cyberphysical systems, and hierarchical and distributed control systems, with applications to unmanned aerial vehicles, distributed robotics, green buildings, and biomolecular networks. Dr. Pappas has received various awards, such as the Antonio Ruberti Young Researcher Prize, the George S. Axelby Award, the Hugo Schuck Best Paper Award, the George H. Heilmeier Award, the National Science Foundation PECASE award and numerous best student papers awards at ACC, CDC, and ICCPS.
\end{IEEEbiography}

\begin{IEEEbiography}
	[{\includegraphics[width=1in,height=1.25in,keepaspectratio]{figs/profile/vpreciado_bw.jpg}}]
	{Victor M. Preciado}
	%received his Ph.D. degree in Electrical Engineering and Computer Science from the Massachusetts Institute of Technology in 2008. He is currently the Raj and Neera Singh Associate Professor of Electrical and Systems Engineering at the University of Pennsylvania. He is a member of the Networked \& Social Systems Engineering (NETS) program, the Warren Center for Network \& Data Sciences, and the Applied Math and Computational Science (AMCS) program. He is a recipient of the 2017 National Science Foundation Faculty Early Career Development (CAREER) Award. His main research interests lie at the intersection of Big Data and Network Science; in particular, in using innovative mathematical and computational approaches to capture the essence of complex, high-dimensional dynamical systems.  Relevant applications of this line of research can be found in the context of socio-technical networks, brain dynamical networks, healthcare operations, biological systems, and critical technological infrastructure.
	received the PhD degree in Electrical Engineering and Computer Science from the Massachusetts Institute of Technology, in 2008. He is currently an Associate Professor of Electrical and Systems Engineering at the University of Pennsylvania, where he is a member of the Networked and Social Systems Engineering (NETS) program, the Warren Center for Network and Data Sciences, and the Applied Math and Computational Science (AMCS) program. He was a recipient of the 2017 National Science Foundation CAREER Award, the 2018 Outstanding Paper Award by the IEEE Control Systems Magazine, and a runner-up of the 2019 Best Paper Award by the IEEE Transactions on Network Science and Engineering. His main research interests lie at the intersection of big data and network science; in particular, in using innovative mathematical and computational approaches to capture the essence of complex, high-dimensional dynamical systems. Relevant applications of this line of research can be found in the context of socio-technical networks, brain dynamical networks, healthcare operations, biological systems, and critical technological infrastructure. He is a Senior Member of IEEE.
\end{IEEEbiography}
%\vfill
%\newpage
\vfill
\fi
}

%\iftoggle{arxiv}{
	\clearpage
	\newpage
	\newpage
	
%=============================================================
\section{\resub{Additional Conditions for Optimality of $\Ponedc$ and $\Ptwodc$}}

\resub{
	% R1-6 a [convergence]
	%%RESUB
	The results established in Theorem~\ref{thm:conv} guarantee that Algorithm~\ref{alg:alg1} will converge to a 
	\ressub{limit}
	value in terms of $\tnn{\Z(W^{(k)},H^{(k)},\Delta^{(k)})}{2n}$.	
	However, because of the non-convexity of $\Ponedc$, such a 
	\ressub{limit}
	value does not need to correspond to its optimal value~$\tnn{\Z(W^{(k)},H^{(k)},\Delta^{(k)})}{2n}=0$, attainable when $\Pone$ is feasible. This fact motivates us to seek additional conditions for optimality of $\Ponedc$ by examining \ressub{limit points associated with the limit values attained by} Algorithm~\ref{alg:alg1} in terms of their Karush-Kuhn-Tucker (KKT) conditions. 
	Next, with this intent, we introduce a standardized version for \ref{pbm:ponesub}.
}

\begin{pbm}{$\Ponestd$}[Standard form for \ref{pbm:ponesub}]
	\label{pbm:ponestd}
	\resub{
		% R1-4 [nonfluent]
		%% RESUB
		This form consists of expressing \ref{pbm:ponesub} in terms of a single unstructured matrix variable~$X\in \R^{4n \times 3n}$, along with affine and semidefinite constraints.	
		Specifically, we introduce the equality constraint $X = \Z(W,H,\Delta)$, along with the reachability constraint~$W \in \spec$ and the structural constraint~$\Delta \in \poly$. Then, we jointly encode these three constraints by an equality constraint~$\A(X)=a_0$ and a semidefinite constraint~$\B(X)\succeq B_0$. Here, 
		$\A: \R^{4n \times 3n}\rightarrow \R^{d_\A}$ and $\B: \R^{4n \times 3n}\rightarrow \Ss^{d_\B}$ are linear operators\footnote{
			\resub{
				%% RESUB
				% R1-5 c [intuition]
				The operator~$\A(X): \R^{4n \times 3n} \rightarrow \R^{d_\A}$ can be concretely expressed as $\A(X) = M\vecop(X)$ for some matrix $M \in \R^{d_\A \times 
					4n\cdot 3n}$.
				The operator~$\B(X)$ can be expressed as 
				$\B(X)= \sum_{i=1}^m \sum_{j=1}^n Q_{i,j} [X]_{i,j}$ for symmetric matrices $\{Q_{i,j} \in \Ss^{d_\B} \}_{i,j=1}^{m,n}$.
		}}
		with ${d_\A}$ and ${d_\B}$ depending on specific forms of $\spec$ and $\poly$. 
		Further, we denote the term $\tr \{ L X  R^\tp\}$ by its inner product representation $\inner{C}{X}$, where $C \coloneqq L^\tp R$.%
	}
	Therefore, the standard form representation of \ref{pbm:ponesub} is described as 
	\begin{align}
	\underset{X}{\minimize}&  \quad \|  X \|_\ast -	\inner{ C}{X}\nonumber\\
	\subjecto & \quad \A(X)=\resub{a_0}, \label{eq:con_std_lin}\\
	& \quad \B(X)\succeq \resub{B_0}.\label{eq:con_std_psd}
	\end{align}
\end{pbm}
 
\begin{lemma}[Optimality conditions for \ref{pbm:ponestd}] 
	\label{lem:kkt_cvx}
	Consider a point~$X^\star$ with rank~$q$ %$q \leq 3n$
	and singular value decomposition
	$U \Sigma V^{\tp}=\svd\{X^\star\}$, where
	$\Sigma = \diag(\sigma_1, \ldots, \sigma_{q}, 0,\ldots, 0) $, 
	$U \in \R^{4n \times 3n}$ with $U=\left [U_q | U_y\right]$,
	$U_q = \left [u_1| \ldots| u_q \right ] $ and  $U_y = \left [u_{q+1}| \ldots| u_{3n}\right ]$, and 
	$V \in \R^{3n \times 3n}$ with $V=\left [V_q| V_y\right ]$,
	$V_q =\left  [v_1| \ldots| v_q\right ] $ and $V_y = \left [v_{q+1}| \ldots| v_{3n}\right ]$.	
	Also, consider the following set, associated with the subdifferential of the nuclear norm of $X$ at $X^\star$:
	\begin{align}
	%	\mathcal Y|_{X^\star} \coloneqq \{Y \in &\R^{4n \times 3n } :\\ &U_q^\tp Y = 0_{q \times 3n}, Y  V_q = 0_{4n \times q},  \| Y \| \leq 1\}.
	\mathcal Y|_{X^\star}\!\! \coloneqq \!\{Y \!\in \R^{4n \times 3n }\! :U_q^\tp Y = 0, Y  V_q = 0,  \| Y \| \leq 1\}.\label{eq:subdiff_set_nuc}	
	\end{align}
	\resub{Further}, let $\mu \in \R^{d_\A}$ and $\Gamma \in \mathbb S^{d_\B}$, be Lagrange multipliers
	%\footnote{ 	\resub{ %% R1-5
	%	For an overview of convex optimization duality, refer to \cite[Ch 5]{boyd2004convex}.}}
	for the constraints associated with operators $\A$ and $\B$, respectively, and define the mapping $G : \R^{d_{\A}} \times \Ss^{d_\B} \rightarrow \R^{4n \times 3n}$ as
	\begin{align*}
	G(\mu, \Gamma) \coloneqq  \A^\ast(\mu)+ \B^\ast(\Gamma).
	\end{align*}
	\resub{Here}, $\A^\ast$ and $\B^\ast$ denote the \resub{adjoint}\footnote{
		\resub{
			% R1-5
			An adjoint operator~$\A^\ast(X)$ with respect to an operator~$\A(X)$ and inner product~$\inner{\cdot}{\cdot}$ is such that $\inner{\A(X)}{a_0}=\inner{X}{\A^\ast(a_0)}$.
	}}
	of their respective operators.
	Then, 
	\ressub{for }
	a primal-dual feasible point $X^\star,(\mu^\star, \Gamma^\star)$ 
	%is
	 \ressub{to be }
	 optimal for \ref{pbm:ponestd} 
	\ressub{it needs to satisfy}
	 the complementary slackness, and, additionally, the Lagrangian stationary condition
	\begin{align}
	\label{eq:kkt_cvx}
	C &+ G(\mu^\star, \Gamma^\star)  = U_qV_q^\tp + Y
	\end{align}
	for some $Y \in \mathcal Y|_{X^\star} $.
\end{lemma}	
\begin{proof}
	Applying the KKT conditions to the convex problem~\ref{pbm:ponestd}, we have that Lagrangian stationarity requires
	{\small
		\begin{align*}
		\nabla_X\left \{\inner{C}{X} + \inner{\mu^\star}{\A(X)-b} + \inner{\Gamma^\star}{\B(X)}\right \}_{X=X^\star} \in \partial \| X^\star\|_\ast,	\end{align*}}%
	where $\partial \| X^\star\|_\ast$ denotes the subdifferential of the nuclear norm at $X^\star$.
	Using the conjugacy property of linear operators and evaluating the gradient of the above equation implies that
	\begin{align*}
	C + \A^\ast(\mu) + \B^\ast(\Gamma) &\in \partial \| X^\star\|_\ast.
	\end{align*}
	Then, using Lemma~\ref{lem:subdiff_nuc} (in the Appendix) for the subdifferential of the nuclear norm, condition~\eqref{eq:kkt_cvx} is obtained. 
\end{proof}

\ressub{
Next, we use the conditions specified above to analyze stationary points of $\Ponedc$, further characterizing such solutions in terms of their optimality.}

\begin{theorem}[\ressub{Optimality conditions 
	for \ref{pbm:ponedc}}]
	\label{thm:stationary}
	%
	%	Consider a primal feasible \emph{stationary} point $\sta X$ generated by Algorithm~\ref{alg:alg1}, with rank $q$ %$q \leq 3n$
	%	and singular value decomposition
	%	$\sta U \sta \Sigma \sta V^{\tp}=\svd\{\sta X\}$ where
	%	$\sta \Sigma = \diag(\sta \sigma_1, \ldots, \sta \sigma_{q}, 0,\ldots, 0) $, 
	%	$\sta U \in \R^{4n \times 3n}$, with $\sta U=\left [\sta U_q | \sta U_y\right]$,
	%	$\sta U_q = \left [\sta u_1| \ldots| \sta u_q \right ] $ and  $\sta U_y = \left [\sta u_{q+1}| \ldots| \sta u_{3n}\right ]$, and 
	%	$\sta V \in \R^{3n \times 3n}$, with $\sta V=\left [\sta V_q| \sta V_y\right ]$,
	%	$\sta V_q =\left  [\sta v_1| \ldots| \sta v_q\right ] $ and $\sta V_y = \left [\sta v_{q+1}| \ldots| \sta v_{3n}\right ]$.
	%	%
	%	Further, consider a point $\sta X$ and its \mbox{dual-feasible} point $(\sta \mu, \sta \Gamma)$ for which the corresponding complementary slackness conditions hold. Also, consider the set
	%	$ \mathcal Y|_{\sta X} \coloneqq \{Y \in \R^{4n \times 3n } : 	\sta U_q^\tp Y = 0, Y  \sta V_q = 0,  \| Y \| \leq 1\}$.
	%Consider a primal feasible \emph{stationary} point~$\sta X$ with rank~$q$ generated by Algorithm~\ref{alg:alg1},
	\ressub{Consider a primal feasible stationary limit point~$\sta X$ for $\Ponedc$}
	along with its 
	singular value decomposition  
	and subdifferential~$\mathcal Y|_{\sta X}$ \resub{(as defined in Lemma~\ref{lem:kkt_cvx})}.
	Further, consider its \mbox{dual-feasible} point $(\sta \mu, \sta \Gamma)$, for which the corresponding complementary slackness conditions hold. 
	%Also, consider the set~\mathcal Y|_{\sta X} \coloneqq \{Y \in \R^{4n \times 3n } : \sta U_q^\tp Y = 0, Y  \sta V_q = 0,  \| Y \| \leq 1\}$.	
	%
	Then, if the Lagrange multipliers~$(\sta \mu, \sta \Gamma)$ are such that 
	\begin{align}
	\label{eq:opt_G}
	G(\sta \mu, \sta \Gamma) \in \mathcal Y_{\sta X},
	\end{align}
	we have that $\sta X, (\sta \mu, \sta \Gamma)$ \ressub{attains} $\tnn{\sta X}{2n}=0$.
\end{theorem}
\begin{proof}	
	The proof is by contradiction. We assume that $q>2n$, implying $\tnn{\sta X}{2n}>0$.
	A limit
	point~$\sta X$ will satisfy $X^{(k+1)}=X^{(k)}=\sta X$ for $k \rightarrow \infty$. Considering the updates performed by Algorithm~\ref{alg:alg1}, we have $L^{(k+1)}=L^{(k)}=\sta L$ and $R^{(k+1)}=R^{(k)}=\sta R$ such that
	$ \sta  L = [\sta  u_1,\ldots, \sta  u_{2n}]^\tp$, and $\sta R = [\sta  v_1,\ldots, \sta  v_{2n}]^\tp$.
	From Lemma~\ref{lem:kkt_cvx}, the Lagrange stationarity condition for $\sta  X$ requires
	\begin{align}
	\label{eq:kkt_stat}
	G(\sta \mu, \sta \Gamma)  =Y +  \sta U_q \sta V_q^\tp - \sta L^\tp \sta R.
	\end{align}
	
	We now split the term $\sta U \sta V^\tp$ as 
	the sum~$\sta U \sta V^\tp =  \sum_{i=1}^{2n} \sta u_i \sta v_i^\tp + \sum_{i=2n+1}^q \sta u_i \sta v_i^\tp$
	and compare it with the product~$\sta L^\tp \sta R =  \sum_{i=1}^{2n} \sta u_i \sta v_i^\tp$, as implied by the stationarity condition. This allows us to rewrite~\eqref{eq:kkt_stat} as
	%\begin{align*}
	$
	G(\sta \mu, \sta \Gamma) =  Y + \sum_{i=2n+1}^q \sta u_i \sta v_i^\tp,
	$
	%\end{align*}
	where the common terms between $\sta U \sta V^\tp$ and $\sta L^\tp \sta R$ have been canceled.
	Since, from Lemma~\ref{lem:kkt_cvx}, the optimality conditions require that $Y \in \mathcal Y_{\sta X}$, we note that any right- or left-singular vectors of $Y$ must be orthogonal to the right- and left-singular vectors appearing in $\sum_{i=2n+1}^q \sta u_i \sta v_i^\tp$.
	Subsequently, as $ G(\sta \mu, \sta \Gamma)$ lies in $\mathcal Y_{\sta X}$ by~\eqref{eq:opt_G}, we must have $Y = G(\sta \mu, \sta \Gamma)$. This implies 
	$ \sum_{i=2n+1}^q \sta u_i \sta v_i^\tp = 0$, which is impossible for $q > 2n$. 
	Therefore, it is the case that $q=2n$, since the minimum rank of $\sta X$ is $2n$, by Theorem~\ref{thm:rank_lyap}. This fact implies that $\sigma_i = 0$ for $i=2n+1,\ldots,3n$, and, consequently, $\tnn{\sta X}{2n} = 0$.
\end{proof}

\ressub{
\begin{remark}
		Conditions under which sequences of points generated by updates as performed by Algorithm~\ref{alg:alg1} will produce limit points that converge to stationary points can be found in Section~3 of \cite{razaviyayn2013unified}.
\end{remark}}

\begin{figure}[!t]
	\centering
	\def\svgwidth{0.45\columnwidth}
	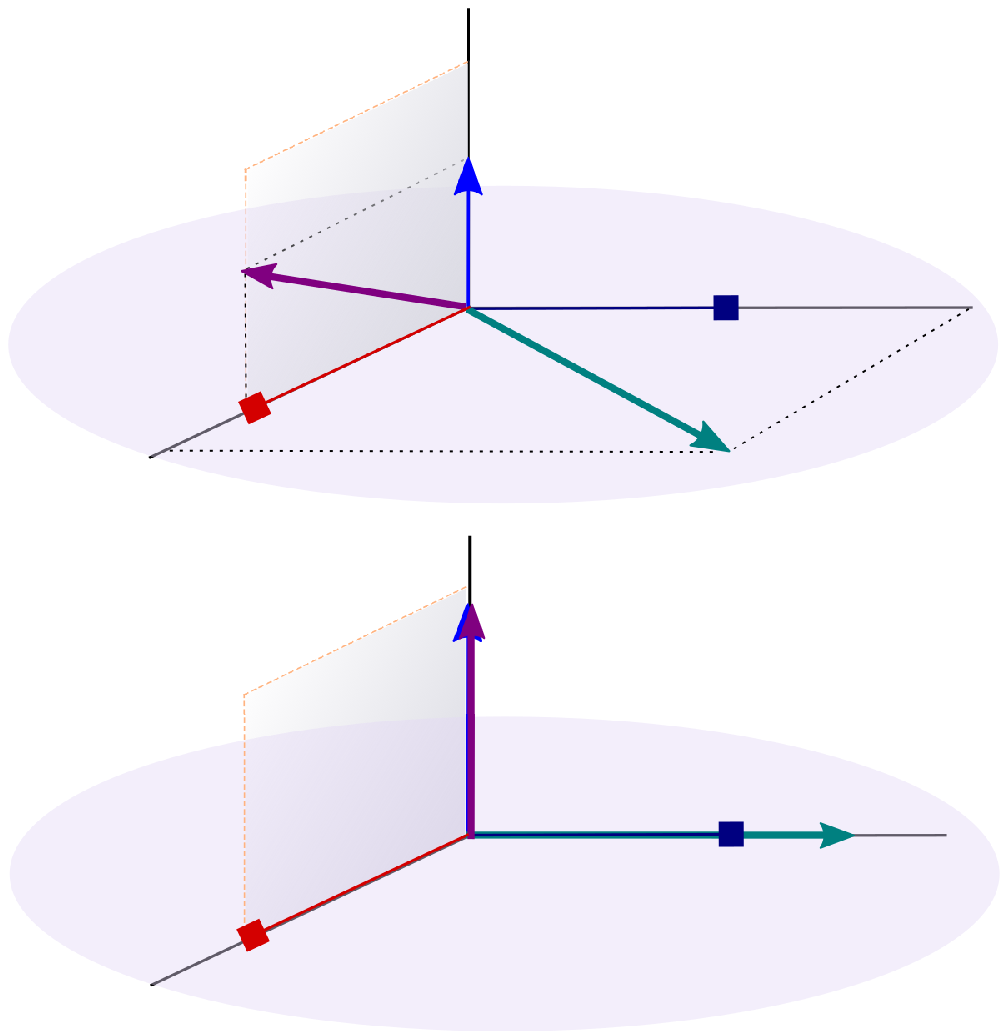	
	\caption{
		\resub{
			% R1-4 c [nonfluent]
			%% 	RESUB
			Geometric representation of the Lagrangian stationarity condition presented in Theorem~\ref{thm:stationary}. In (a), the case where optimality is not attained, since~$G(\sta \mu,\sta \Gamma) \neq Y$ for $Y \in \mathcal Y_{\sta X}$. Here, $\sta X$ has a component in the span of $\sta  U_{2n+1:q}\sta V_{2n+1:q}^\tp= \{ \sta u_i \sta v_i^\tp\}_{i=2n+1}^q$, therefore, $\tnn{\sta X}{2n} > 0$.
			In (b), the case where optimality is attained, with $G(\sta \mu,\sta \Gamma) = Y$ for $Y \in \mathcal Y_{\sta X}$. Here, $\sta X$ is in the span of $\sta U_{1:2n}\sta V_{1:2n}^\tp = \{ \sta u_i \sta v_i^\tp\}_{i=1}^{2n}$, which implies that $\tnn{\sta X}{2n}=0$.
		}
	}
	\label{fig:kkt}
\end{figure}

\resub{
	%% R1-5 c [intuition]
	%
	We can interpret Theorem~\ref{thm:stationary} to get some intuition on the conditions for optimality of $\Ponestd$.
	% at stationarity.
	%
	First, we recall that primal feasibility requires constraints~\eqref{eq:con_std_lin} and \eqref{eq:con_std_psd} to be satisfied at $\sta X$.
	Next, we notice from \eqref{eq:subdiff_set_nuc} that, for optimality, the subdifferential set~$\mathcal Y|_{\sta X}$ is required to be orthogonal to the row and column spaces of  primal-feasible point~$\sta X$.
	Thus, Theorem~\ref{thm:stationary} shows us how the set~$\mathcal Y|_{\sta X}$ restricts the space for dual feasibility of $\Ponestd$. Specifically, it requires the existence of Lagrange multipliers~$\sta \mu$ and $\sta \Gamma$ such that $G(\sta \mu, \sta \Gamma)$ becomes contained in the low-dimensional space defined by~$\mathcal Y|_{\sta X}$.
	Simply stated, the higher the dimension of the space required for primal feasibility -- as induced by the structural and reachability constraints -- the more restricted becomes the space~$\mathcal Y_{\sta X}$ available for dual feasibility (and hence, for  optimality).%
}
A geometrical representation of the stationary conditions for optimality obtained in Theorem~\ref{thm:stationary} is presented in Figure~\ref{fig:kkt}.

Next, in line with the result presented in Theorem~\ref{thm:stationary} for \ref{pbm:ponedc}, we derive a Lagrangian stationarity condition for a 
limit
point~$\sta X$ now associated with a particular value of $\gamma$, to satisfy $\tnn{\sta X}{2n}\! =\!0$ for \ref{pbm:ptwodc}.  

\newcommand{\Ld}{\mathcal{L}_{^{_\Delta}}}
\newcommand{\DeltaM}{\Delta^{\!\llcorner}}
\newcommand{\DeltaMsta}{{\sta\Delta}^{\!\llcorner}}
\newcommand{\Y}{\mathcal Y}
\newcommand{\F}{\mathcal F}

\newcommand{\YSetSta}{\mathcal Y|_{\sta X}}
\newcommand{\FSetSta}{\mathcal F|_{\sta X}}

\begin{theorem}[Optimality conditions
	 %at stationarity 
	 for \ref{pbm:ptwodc}]
	\label{thm:global_pen}
	Consider a primal feasible
	 stationary limit
	 point~$\sta X = \Z(\sta W,\sta H,\sta \Delta)$ 	
	 for some $\gamma$, with $\rank[\sta X]=q$ and singular value decomposition as described in Theorem~\ref{thm:stationary}.	
	%	$\sta U \sta \Sigma \sta V^{\tp}=\svd\{\sta X\}$ with
	%	$\sta \Sigma = \diag(\sta \sigma_1, \ldots, \sta \sigma_{q}, 0,\ldots, 0) $, 
	%	$\sta U \in \R^{4n \times 3n}$, where $\sta U=\left [\sta U_q | \sta U_y\right]$,
	%	$\sta U_q = \left [\sta u_1| \ldots| \sta u_q \right ] $,  $\sta U_y = \left [\sta u_{q+1}| \ldots| \sta u_{3n}\right ]$, and with
	%	$\sta V \in \R^{3n \times 3n}$, where $\sta V=\left [\sta V_q| \sta V_y\right ]$,
	%	$\sta V_q =\left  [\sta v_1| \ldots| \sta v_q\right ] $, and $\sta V_y = \left [\sta v_{q+1}| \ldots| \sta v_{3n}\right ]$.	
	%	%
	Also, consider its dual-feasible point~$(\sta \mu, \sta \Gamma)$, for which the corresponding complementary slackness conditions hold. 
	Further, define the matrix \vspace{-1mm}
	\begin{align*}
	\DeltaMsta =
	\begin{bmatrix}
	0 & \sta \Delta \\  0 & 0
	\end{bmatrix} \in \R ^{4n \times 3n}.
	\end{align*}
	If the Lagrange multipliers $(\sta \mu, \sta \Gamma)$ associated with the primal-dual feasible limit point~$\sta X, (\sta \mu, \sta \Gamma)$ are such that 
	\begin{align}
	G(\sta \mu, \sta \Gamma) - \gamma \sgn(\DeltaMsta) = Y + F,
	\end{align}
	where $Y\in \YSetSta$ and $F$ is in the set
	$
	\FSetSta \!\coloneqq \{  F \in \R^{4n \times 3n}\!: [F]_{i,j} \!= \!0 \text{ if } [\sta X]_{i,j}\! \neq 0 , \| F\|_\infty \leq \gamma\},
	$
	%	where, for $F$ we have
	%	\begin{align*}
	%	[F]_{i,j} = 0 \text{ if } [\sta X]_{i,j} \neq 0, \text{ and } \| F\|_\infty \leq \gamma,
	%	\end{align*}	
	%	and, for $Y$ we have
	%	\begin{align*}
	%	\sta U_q^\tp Y = 0,\; Y  \sta V_q = 0,\text{ and } \| Y \| \leq 1.
	%	\end{align*}
	%
	then, we have that $\sta X$ attains $\tnn{\sta X}{2n}\!=\!0$ for \ref{pbm:ptwodc}. Here,
	$\sgn(\Delta)$ applies the signum function over each entry of $\Delta$, evaluating to $[\sgn(\Delta)]_{i,j}=1$ if $[\Delta]_{i,j}>0$,  $[\sgn(\Delta)]_{i,j}=-1$ if $[\Delta]_{i,j}<0$, and $[\sgn(\Delta)]_{i,j}=0$, otherwise.	
\end{theorem}

\begin{proof}
	We define the linear operator $\Ld: \R^{4n \times 3n} \rightarrow \R^{4n \times 3n}$, which extracts $\Delta$ from the upper-right block of $\sta X$, such that $\Ld(X;A) =  \DeltaM$.
	In particular, we have $\| \Ld(X;A) \|_1 = \| \Delta \|_1$.
	This operator allows \ref{pbm:ptwodc} to be expressed in the standard form, with objective function written in terms of the single variable $X \in \R^{4n \times 3n}$ and convex constraints imposed by the linear operators $\A(X)=a_0$ and $\B(0) \succeq B_0$, to which we associate the function $G(\mu, \Gamma) \coloneqq  \A^\ast(\mu)+ \B^\ast(\Gamma)$.
	By Lemmas~\ref{lem:subdiff_l1} and \ref{lem:subdiff_nuc} in this section of the Appendix, the Lagrangian stationarity condition for the standard form associated with \ref{pbm:ptwodc}, at the primal-dual feasible point $\sta X,(\sta \mu,\sta \Gamma)$, requires
	\begin{align*}
	\sta U \sta V^\tp + Y - \sta L^\tp \sta R  + \gamma\sgn(\DeltaMsta) + F - G(\sta \mu, \sta \Gamma) =0,
	\end{align*}
	where $Y \in \YSetSta$ and $F \in \FSetSta$.
	By splitting $\sta U \sta V^\tp$ as the sum
	$\sum_{i=1}^r \sta u_i \sta v_i^\tp + \sum_{i=r+1}^q \sta u_i \sta v_i^\tp$, we have
	\begin{align*}
	\sum_{i=r+1}^q \sta u_i \sta v_i^\tp  + Y + \gamma \sgn(\DeltaMsta) +  F - G(\mu, \Gamma) = 0,
	\end{align*}
	where the common terms between $\sta U \sta V^\tp$ and $\sta L^\tp \sta R$ have been canceled.
	Therefore, $\tnn{\sta X}{2n}=0$ will be attained if $\sum_{i=r+1}^q \sta u_i \sta v_i^\tp = 0$. By letting $Y$ and $F$ such that
	\begin{align*}
	Y + F = G(\sta \mu, \sta \Gamma) - \gamma \sgn(\DeltaMsta),
	\end{align*}
	the desired condition is achieved.
\end{proof}

\subsection{Additional Lemmas}

\begin{lemmaa}[Subdifferential of matrix $1$-norm]
	\label{lem:subdiff_l1}
	Let $X \in \R^{m \times n}$, and denote by $\sgn(X)$ the matrix containing the result of the signum function applied at each entry of $X$. Then, the subdifferential of $\| X\|_1$ is given by
	$
	\partial \| X \|_1 = \{ \sgn(X) + F : F \in \R^{m \times n}, 
	[F]_{i,j} = 0 \text{ if } [X]_{i,j} \neq 0 , \| F\|_\infty \leq 1\},
	$
	where $i=1,\ldots,m$, $j=1,\ldots,n$, and $\| F\|_\infty= \max_{i,j} [F]_{i,j}$. 
\end{lemmaa}
%Alternatively, the subdifferential can be characterized by
%\begin{align*}
%\partial \| X \|_1 &= \{ \sgn(X) + F : F \in \R^{m \times n}, \Po(F) = 0, \| F\|_\infty \leq 1\},
%\end{align*}
%where $\Po(F)$ is the projection operator of $F$ on $\Omega = \supp(X)$, given by\vspace{-5mm}
%\begin{align*}
%[\Po(F)]_{i,j} =
%\begin{cases}
%[F]_{i,j} &(i,j) \in \Omega,\\
%\;0&\text{otherwise}.
%\end{cases}
%\end{align*}
%The projection operator on the orthogonal space complementary to $\Omega$ is given by\vspace{-5mm}
%\begin{align*}
%[\Pop(F)]_{i,j} =
%\begin{cases}
%\;0 &(i,j) \in \Omega,\\
%[F]_{i,j}&\text{otherwise}.
%\end{cases}
%\end{align*}
\begin{proof}
	See \cite[p.244]{calafiore2014optimization}.
\end{proof}

\begin{lemmaa}[Subdifferential of matrix nuclear norm]
	\label{lem:subdiff_nuc}
	Let $X \in \R^{m \times n}$ with rank $q$, and singular value decomposition 
	\resub{
		% R1-2 notation (compact)
		$X = U_q S_q V_q^\tp$, 
		with $S_q=\diag(\sigma_1, \ldots, \sigma_q)$,
		$U_q \in \R^{m \times q}$, and $V_q \in \R^{n \times q}$.
	}
	Then, the subdifferential of $\| X\|_\ast$ is given by
	$ \partial \| X \|_\ast = \{ U_qV_q^\tp + Y : Y \in \R^{m \times n}, 
	YV_q = 0, U_q^\tp Y = 0, \| Y\|\leq 1\}$,	
	where $\| Y \|$ denotes the operator norm of $Y$.
\end{lemmaa}
%Alternatively, the subdifferential can be characterized as
%\begin{align*}
%\partial \| X_0 \|_\ast &= \{ UV^\tp + Y : Y \in \R^{m \times n}, \Pt(Y) = 0, \| Y\|\leq 1\} 
%%&\qquad \qquad \qquad  YV = 0, U^\tp Y = 0, \| Y\|\leq 1\} 
%\end{align*}
%where $\Pt(Y)$ is the projection of $Y$ on the linear space
%\begin{align*}
%T = \{UX_1^\tp +X_2V^\tp: X_1\in \R^{n \times q}, X_2 \in \R^{m \times q} \}.
%\end{align*}
%Explicitly, this projection operator can be computed as
%\begin{align*}
%\Pt(Y) = UU^\tp Y + YVV^\tp  +  UU^\tp Y VV^\tp,
%\end{align*}
%while the projection operator onto the complementary space orthogonal to $T$ is given by
%\begin{align}
%\Ptp(Y) = (I - UU^\tp)Y(I - VV^\tp).
%\end{align}
\begin{proof}
	See \cite{watson1992characterization} and \cite[p.481]{recht2010guaranteed}.
\end{proof}
%}

\end{document}